\documentclass[12pt,reqno,amscd,amssymb,pstricks,latexsym,amsbsy,xypic,mathrsfs,verbatim]{amsart}
\usepackage{amsfonts,amssymb,amsthm}
\usepackage{amsmath,amscd}
\usepackage{pstricks}
\usepackage{pst-node}
\usepackage{mathrsfs}
\usepackage[all]{xy}
\usepackage[enableskew,vcentermath]{youngtab}
\usepackage{hyperref}

\renewcommand{\subsection}[1]{\vspace{.18in}\par\noindent\addtocounter{subsection}{1}
\setcounter{equation}{0}{\bf\thesubsection.\hspace{5pt}#1}}

\newtheorem{theorem}{Theorem}[section]
\numberwithin{equation}{theorem}

\theoremstyle{definition}
\newtheorem{Example}[theorem]{Example}%[section]
\newtheorem{Examples}[theorem]{Examples}
\newtheorem{Rem}[theorem]{Remark}

\theoremstyle{plain}
\newtheorem{Prop}[theorem]{Proposition}
\newtheorem{Thm}[theorem]{Theorem}

\newtheorem{Lem}[theorem]{Lemma}
\newtheorem{Coro}[theorem]{Corollary}

\newcommand{\scr}[1]{\mathscr #1}

\def\lan{\langle}    \def\ran{\rangle}
\def\lr#1{\langle #1\rangle}

\def\az{\alpha}  
\def\bz{\beta}  \def\dt{\Delta}

\def\ggz{\Gamma}  
\def\sz{\sigma}

\def\vez{\varepsilon}  
\def\dz{\delta}

\def\lz{\lambda}
      \def\cF{{\mathcal F}}
\def\cC{{\mathcal C}}  \def\cX{{\mathcal X}}   
\def\cH{{\mathcal H}}

  \def\cF{{\mathcal F}}  \def\cT{{\mathcal T}}
\def\cA{{\mathcal A}}    
  \def\scrC{{\scr C}}  \def\scrR{{\scr R}}
   \def\scrE{{\scr E}}  
 
\def\bbN{{\mathbb N}}  \def\bbZ{{\mathbb Z}}  
  \def\bbI{{\mathbb I}}  \def\bbF{{\mathbb F}}
\def\bbC{{\mathbb C}}    \def\bbP{{\mathbb P}}
\def\bbA{{\mathbb A}}  \def\scrR{{\scr R}}
         
      \def\vphi{\varphi}
        \def\bfd{{\bf d}} 
   
  \def\leq{\leqslant}  \def\geq{\geqslant}
 
\def\lra{\longrightarrow}   
\def\lmto{\longmapsto}   
\def\ra{\rightarrow}
\def\Hom{\mbox{\rm Hom}}  
       \def\im{\mbox{\rm Im}\,}
\def\Ext{\mbox{\rm Ext}\,}   \def\Ker{\mbox{\rm Ker}\,}
\def\dim{\mbox{\rm dim}\,}   \def\End{\mbox{\rm End}}
\def\Aut{\mbox{\rm Aut}}     \def\id{\mbox{\rm id}}
\def\udim{{\mathbf{ dim\,}}} 
\def\mod{\mbox{\rm -mod}}  \def\tr{\mbox{\rm tr}}
  
\def\rep{\mbox{\rm Rep}\,}  
\def\reg{\mbox{\rm Reg}\,}  \def\prim{{\rm prim}}
\def\fkg{{\frak g}}

\def\fkO{{\frak O}}

  \def\bfd{{\bf d}}

\def\supp{{\rm supp}} \def\mult{{\rm mult}}

\def\wt{\widetilde}

\def\Rep{{\rm Rep}}
\def\wh{\widehat}

\def\bfd{{\bf d}} \def\bfx{{\bf x}}  \def\bfy{{\bf y}}

\def\spec{{\rm Spec}}
\def\add{{\rm add\,}}
\def\tail{{\rm t}} \def\head{{\rm h}}

\def\one{{\boldsymbol 1}}
\def\ofq{{\overline{\bbF}}_q}
\def\re{{\rm re}} \def\im{{\rm im}}
\def\reg{{\rm reg}}
\def\GL{{\rm GL}}

\begin{document}

\title[Primitive elements in Ringel--Hall algebras]
{Primitive elements in Ringel--Hall algebras \\ of tame hereditary algebras}
\author{Bangming Deng and Weihao Li}
\address{Department of Mathematical Sciences, Tsinghua University,
Beijing 100084,  China.} \email{bmdeng@tsinghua.edu.cn }
\email{li-wh20@mails.tsinghua.edu.cn}
%\address{Department of Mathematical Sciences, Tsinghua University,
%Beijing 100084,  China.} \email{li-wh20@mails.tsinghua.edu.cn}

%\thanks{Supported by the Natural Science Foundation of China (Grant No. 12031007).}

%\date{\today}

\subjclass[2000]{16G20, 16G60, 17B37}

\begin{abstract}
We study primitive elements in the Ringel--Hall algebra $\cH(A)$
of an algebra $A$ over a finite field associated with a quiver with automorphism.
When $A$ is a tame hereditary algebra, we give a description
of primitive elements in $\cH(A)$ which generalizes and improves a result
of Hennecart (IMRN 2021) \cite{Hen} for tame quivers. Moreover, we
obtain an identity concerning primitive elements in the subalgebra of $\cH(A)$
generated by regular $A$-modules which enables us to construct an explicit
basis for the space of primitive elements in $\cH(A)$.
\end{abstract}
\maketitle

\section{Introduction}

The classical Hall algebra introduced by Steinitz \cite{St} and later by Hall
\cite{Ha} is isomorphic to the ring of symmetric functions
and plays an important role in the study of representations of symmetric
groups and general linear groups; see \cite{Mac}. It is a graded Hopf algebra
generated by primitive elements.

In 1990s, Ringel \cite{R90a} generalized the approach of Steinitz and Hall and
constructed an associative algebra from the category of finite length modules
over a finitary ring. It turns out that these algebras, called Ringel--Hall
algebras nowadays, have a deep connection with Kac--Moody Lie algebras and
their quantized enveloping algebras. In fact, by results of Ringel \cite{R90c} and
Green \cite{Gr95}, for a finite dimensional hereditary algebra $A$ over a finite
field $k=\bbF_q$, the subalgebra of its Ringel--Hall algebra $\cH(A)$
generated by simple modules is isomorphic to the positive part of the quantized
enveloping algebra $U_{\boldsymbol v}({\frak g})$ of the (symmetrizable) Kac--Moody Lie algebra
$\frak g$ associated with $A$ by specializing $\boldsymbol v$ to $\sqrt{q}$.

It was further shown by Sevenhant and Van den Bergh \cite{SVd2} (see also \cite{DX3})
that the whole algebra $\cH(A)$ is generated by its primitive elements and is
isomorphic to the positive part of the quantized enveloping algebra of a generalized
Kac--Moody Lie algebra in the sense of Borcherds \cite{Bor}.
Inspired by \cite{SVd2}, Deng and Xiao \cite{DX3} proved that the dimension
of the space of primitive elements in $\cH(A)$ of a fixed dimension vector is a
polynomial in $q$ with rational coefficients, and related them with Kac's conjecture,
which is concerned with numbers of absolutely indecomposable representations
of quivers. The generating functions of these polynomials have been
studied in \cite{BS}, where primitive elements are called cuspidal functions,
and the dimension of absolutely cuspidal functions of a fixed dimension vector
was introduced and studied.

In \cite{BG}, the authors proved that the Hall algebra of every profinitary
and cofinitary exact category, including the category of finite dimensional
 modules over each finitely generated $\bbF_q$-algebra, is generated by
its primitive elements. It is thus interesting to describe primitive elements
in a given Ringel--Hall algebra.

For a positive integer $r$, let $C_r$ denote the cyclic quiver with $r$ vertices.
By $\Rep_{\bbF_q}C_r$ (resp., $\Rep^0_{\bbF_q}C_r$) we denote the category of finite
dimensional (resp., nilpotent) representations of $C_r$ over $\bbF_q$.
The Ringel--Hall algebra $\cH(\bbF_q C_r)$ of the path algebra $\bbF_q C_r$ of
$C_r$ admits a subalgebra $\cH^0(\bbF_q C_r)$ generated by
nilpotent representations of $C_r$ over $\bbF_q$ which was widely studied;
see, e.g., \cite{R93a,L99,VV,Sch,DD05,DDF}. When $r=1$, i.e., $C_1$ is the
Jordan quiver, $\cH^0(\bbF_q C_1)$ is exactly the classical Hall
algebra whose primitive elements can be explicitly
constructed; see \cite[Ch.~III, Sect.~7]{Mac}. Namely, let $S$ be the unique
simple nilpotent representation of $C_1$ over $\bbF_q$, and for $t\geq 1$,
let $S[t]$ denote the indecomposable nilpotent representation of $C_1$
of dimension $t$. For each partition $\lz=(\lz_1,\ldots,\lz_m)$, set
\begin{equation} \label{indec-lambda}
I_\lz=\bigoplus_{i=1}^m S[\lz_i].
\end{equation}
Then the elements
$$p_n^{(1)}=p_n=\sum_{\lz \vdash n}\bigg(\prod_{s=1}^{\ell(\lz)-1}(1-q^s)
\bigg)[I_\lz],\;\forall\, n\geq 1$$
are primitive and form a minimal set of generators in $\cH^0(\bbF_q C_1)$,
where $\lz$ is a partition of $n$, $\ell(\lz)$ is the length of $\lz$, and
$[I_\lz]$ denotes the isoclass (isomorphism class) of $I_\lz$.

In general, based on Schiffmann's work \cite{Sch}, Hubery \cite{Hub1}
constructed a family of central elements $\cH^0(\bbF_q C_r)$ and derived
primitive elements $p_n^{(r)}$ admitting expression
$$p_n^{(r)}=\sum_{i=1}^r [M_i]+\text{other terms},$$
where $M_1,\ldots,M_r$ are all pairwise non-isomorphic indecomposable
nilpotent representations of dimension $nr$ of $C_r$. A formula
for $p_n^{(r)}$ in terms of monomials of these central elements was
given in \cite{Ma}.

Recently, Hennecart \cite{Hen} studied primitive elements in the Ringel--Hall algebra
$\cH(A)$ of the path algebra $A=\bbF_q Q$ of a tame quiver $Q$ and proved that
the space of primitive elements in $\cH(A)$ with dimension vector a fixed positive
 imaginary root is the kernel of a nonzero linear function on the space of {\it regular}
primitive elements with the same dimension vector. The proof is based on the description
of primitive elements in the cyclic quiver case, as well as the well-known
embedding from the category of finite dimensional representations of the Kronecker
quiver into that of finite dimensional representations of $Q$.

The main purpose of the present paper is to apply the approach of Hennecart
to obtain a more explicit description of primitive elements in the
Ringel--Hall algebra of an arbitrary tame hereditary $\bbF_q$-algebra.
To this aim, we deal with the Ringel--Hall algebra $\cH(A)$ of an $\bbF_q$-algebra
$A={\frak A}(Q,\sz;q)$ associated with a quiver $Q$ with automorphism $\sz$ introduced in
\cite{DD06}; see the definition in Section 2. Let $A\mod$ denote the category
of finite dimensional $A$-modules. By the definition, $\cH(A)$ is
a complex vector space with a basis given by the isoclasses
$[M]$ of modules $M\in A\mod$, and it is a (twisted) bialgebra whose
multiplication and comultiplication $\dt$ are adjoint with respect to the
symmetric bilinear form
$\{-,-\}: \cH(A)\times \cH(A)\ra \bbC$ defined by
$$\{[M],[N]\}=\delta_{[M],[N]}\frac{1}{a_M},\;\forall\,M,N\in A\mod$$
 where $a_M$ is the cardinality
 of the automorphism group $\Aut(M)$, that is, there holds that
$$\{xy,z\}=\{x\otimes y, \dt(z)\},\,\forall\, x,y,z\in \cH(A).$$
 This fact indeed implies that $\cH(A)$ is generated by primitive elements.

Now let $Q$ be an acyclic tame quiver, that is, $A={\frak A}(Q,\sz;q)$ is a tame hereditary
algebra, and let $I$ denote the set of isoclasses of simple $A$-modules. For each
$M\in A\mod$, by $\udim M$ we denote the dimension vector of $M$.
Then $\cH(A)$ is $\bbN I$-graded with
$$\cH(A)=\bigoplus_{\bfd\in\bbN I}\cH(A)_{\bfd},$$
where $\cH(A)_{\bfd}$ is spanned by $[M]$ with $\udim M=\bfd$.
Let $\dz$ be the minimal positive imaginary root in the root system associated
with $(Q,\sz)$, and
for each $n\geq 1$, set
$$\cH(A)_{n\dz}^{\prim}=\{x\in \cH(A)_{n\dz}\mid \dt(x)=x\otimes 1+1\otimes x\}.$$

It is known that the full subcategory $\scrR_A$ of $A\mod$ consisting of
regular $A$-modules is closed under extensions. This defines a subalgebra
$\cH(\scrR_A)$ of $\cH(A)$. Further, the comultiplication of $\cH(A)$ restricts
to a comultiplication $\dt_\scrR$ on $\cH(\scrR_A)$ such that
$\cH(\scrR_A)$ becomes a bialgebra. Further, set
$$\cH(\scrR_A)_{n\dz}^{\prim}=\{x\in \cH(\scrR_A)_{n\dz}\mid
\dt_\scrR(x)=x\otimes 1+1\otimes x\}.$$
 A slight modification of \cite[Props.~6.1\;\&\;6.2]{Hen}
gives that
$$\cH(A)_{n\dz}^{\prim}\subset \cH(\scrR_A)_{n\dz}^{\prim}.$$

As in \cite{Hen}, set
$$\one_{n\dz}^\reg=\sum_{{[M],\,M\in\scrR_A}\atop{\udim M=n\dz}}[M]\in \cH(A)_{n\dz}$$
and define a linear function
$$\bbI^\reg_{n\dz}=\{-,\one_{n\dz}^\reg\}: \cH(\scrR_A)^{\prim}_{n\dz}\lra \bbC,\;
z\lmto \{z,\one_{n\dz}^\reg\}.$$

One of our main purposes is to prove the following theorem which generalizes
\cite[Thm~1.1]{Hen} to an arbitrary tame hereditary $\bbF_q$-algebra.

\begin{Thm} \label{Main-thm-1} Let $Q$ be an acyclic tame quiver with
automorphism $\sz$. Then for each $n\geq 1$,
$$\cH(A)_{n\dz}^{\prim}=\Ker \bbI^\reg_{n\dz}.$$
\end{Thm}

The theorem is indeed a consequence of the formula
\begin{equation} \label{key-identity-1}
\bigg\{p_n^{(r)},\one_{n\dz}^{(r)}\bigg\}
=\sum_{\lz \vdash n}\frac{\prod_{s=1}^{\ell(\lz)-1}(1-q^s)}{a_\lz(q)}\;\;(n\geq 1)
\end{equation}
 in $\cH^0(\bbF_q C_r)$, where $a_\lz(q)$ is the cardinality of the automorphism
 group of $I_\lz$, and
$$\one_{n\dz}^{(r)}=\sum_{{[M],\,M\in \Rep^0_{\bbF_q}C_r}\atop{\udim M=n\dz}}[M].$$
This formula is clearly true for $r=1$. When $r>1$, it can be obtained by
viewing $\cH^0({\bbF_q}C_r)$ as a subalgebra of $\cH(\bbF_q Q)$
for a specific acyclic tame quiver $Q$ of type $\wt A$.

Furthermore, by applying Fourier transforms of Ringel--Hall algebras, we obtain
the identity
\begin{equation} \label{key-identity-2}
\sum_{\lz \vdash n}\frac{\prod_{s=1}^{\ell(\lz)-1}(1-q^s)}{a_\lz(q)}
=\frac{1}{q^n-1},\;\,\forall\,n\geq 1,
\end{equation}
which indeed enables us to give an explicit basis of $\cH(A)_{n\dz}^{\prim}$.
Let us explain this in the case where $A=\bbF_q Q$ for
an acyclic tame quiver $Q$ (with $\sz=\id$). In this case, $\scrR_A$
consists of infinitely many tubes which can be parameterized by the
projective line $\bbP^1_{\bbF_q}$. For each $x\in \bbP^1_{\bbF_q}$,
let $\deg(x)$ denote its degree and $\cT_x$ be the corresponding
tube in $\scrR_A$. Then each $\cT_x$ defines a subalgebra $\cH(\cT_x)$
of $\cH(\scrR_A)$ which is isomorphic to $\cH^0(\bbF_{q^d} C_r)$,
where $d=\deg(x)$ and $r$ is the rank of $\cT_x$. Then for each $m\geq 1$,
the primitive element $p_m^{(r)}\in\cH^0(\bbF_{q^d} C_r)$ defined above
(over $\bbF_{q^d}$) gives rise to a primitive element
$$p_m(x)\in \cH(\scrR_A)_{md\dz}^{\prim},$$
which satisfies that
$$\big\{p_m(x),\one^\reg_{md\dz}\big\}=\sum_{\lz \vdash m}
\frac{\prod_{s=1}^{\ell(\lz)-1}(1-q^{ds})}{a_\lz(q^d)}.$$
 This together with the identity \eqref{key-identity-2} gives the following result
which extends \cite[Cor.~6.10]{Hen} in a full generality.

\begin{Thm} \label{Main-thm-2} Let $Q$ be an acyclic tame quiver with $A=\bbF_q Q$,
and for each $n\geq 1$, fix $x\in \bbP_{\bbF_q}^1$ with $n=m\deg(x)$ for some
$m\geq 1$. Then the set
$$\bigg\{ p_m(x)-p_t(y)\;\bigg|\; y\in \bbP_{\bbF_q}^1,\; y\not=x,\;
\text{ and }\;\,t\deg(y)=n\bigg\}$$
is a basis of $\cH(A)^{prim}_{n\dz}$.
\end{Thm}

The paper is organised as follows. Section 2 gives a brief introduction on the
Ringel--Hall algebra $\cH_v(Q,\sz)$ associated with a quiver $Q$ with automorphism $\sz$
and its primitive elements. In section 3, we relate primitive elements in $\cH_v(Q,\sz)$
with those in its subalgebras defined by full subcategories of $A$-modules, with
emphasis on the full subcategory of regular $A$-modules, as well as that
arising from an orthogonal exceptional pair of $A$-modules. Section 4 recalls a
construction of primitive elements in the Ringel--Hall algebra of nilpotent representations
of a cyclic quiver, including the classical Hall algebra. In section 5 we describe
regular primitive elements in $\cH_v(Q,\sz)$ when $Q$ is an acyclic tame quiver.
The proof of Theorem \ref{Main-thm-1} is given in Section 6, and Theorem \ref{Main-thm-2}
follows from Theorem \ref{Main-thm-1} and the identity in
\eqref{key-identity-2}, whose proof involves Fourier transforms of Ringel--Hall
algebras and is given in Section 7.

\medskip

{\it Throughout the paper, $\bbF_q$ denotes a finite field of $q$ elements,
$A$ often denotes an $\bbF_q$-algebra, and let $A\mod$ be the category
of finite dimensional (left) $A$-modules. For a module $M\in A\mod$, by $[M]$
we denote the isoclass of $M$, and for each positive integer $s\geq 1$,
we write
$$sM=\underbrace{M\oplus \cdots\oplus M}_s.$$
For a class $\mathcal X$ of $A$-modules, let $\add\cX$ denote the full
subcategory of $A\mod$ consisting of the direct summands of finite direct sums
of $A$-modules in $\mathcal X$. }

\medskip

\section{Ringel--Hall algebras and their primitive elements}

In this section we recall the definition of the (twisted) Ringel--Hall algebra of
an algebra $A$ over a finite field and state some properties of their primitive elements
when $A$ is associated with a quiver with automorphism.

Let $A$ be a finitely generated algebra over the finite field $\bbF_q$.
For $L, M_1,\ldots,M_t\in A\mod$,
let $F^L_{M_1,\ldots,M_t}$ be the number of the filtrations of submodules
$$L=L_0\supseteq L_1\supseteq \cdots\supseteq L_{t-1}\supseteq L_t=0$$
 such that $L_{i-1}/L_i\cong M_i$ for all $1\le i \le t$. Since
$L$ is finite as a set, $F^L_{M_1,\ldots,M_t}$ is a nonnegative
integer, called the {\it Hall number} associated with
$L,M_1,\ldots, M_t$. In particular, if $t=2$, $F^L_{M_1,M_2}$ is the number of
submodules $X$ of $L$ such that $X\cong M_2$ and $L/X\cong M_1$. For each $M\in A\mod$,
set $a_M=|\Aut_A(M)|$, the cardinality of the automorphism group of $M$.

Following Ringel \cite{R90a}, the Hall algebra $H(A)$ of $A$ is by definition the algebra
over the complex field $\bbC$ with a basis $\{[M]\mid M\in A\mod\}$,
the set of isoclasses of modules in $A\mod$, and with multiplication given by
\begin{equation}\label{twisted multn}
[M]\diamond [N]=\sum_{[L]} F^L_{M,N}[L],\;\,\forall\, M,N\in A\mod,
\end{equation}
 where the sum is over all the isoclasses of $A$-modules $L$. This is
indeed a finite sum since $A$ is finitely generated. Then
$H(A)$ is an associative algebra with identity $[0]$, the isoclass
of zero module. Moreover, by a result of Green \cite{Gr95}, $H(A)$ admits a
coalgebra structure with comultiplication given by
$$\dt([M])=\sum_{[X],[Y]}\frac{a_{X} a_{Y}}{a_M} F_{X,Y}^M\,  [X]\otimes [Y],\;\,\forall\, M\in A\mod.$$
As usual, an element $x\in H(A)$ is called primitive if
$$\dt(x)=x\otimes 1+1\otimes x.$$
 The following result is a consequence of \cite[Main Thm 1.2]{BG}, which generalizes
a result of Sevenhant--Van den Bergh \cite{SVd1} for path algebras of quivers.

\begin{Prop} Let $A$ be a finitely generated algebra over a finite field $\bbF_q$.
Then $H(A)$ is generated by its primitive elements.
\end{Prop}

In this paper we mainly focus on finitely generated $\bbF_q$-algebras
associated with quivers with automorphism and their Ringel--Hall
algebras with a twisted multiplication via associated Euler forms.

Let $Q=(Q_0,Q_1)$ be a finite quiver with vertex set $Q_0$ and arrow
set $Q_1$. For each arrow $\rho$ in $Q_1$, we denote by $h\rho$ and
$t\rho$ the head and the tail of $\rho$, respectively.
Let $\sz$ be an automorphism of $Q$, that is, $\sz$ is a permutation on $Q_0$,
as well as on $Q_1$, satisfying $\sz(h\rho)=h\sz(\rho)$ and $\sz(t\rho)=t\sz(\rho)$
for each $\rho\in Q_1$.

Let $\ofq$  be the algebraic closure of $\bbF_q$ and $\ofq Q$ be the path algebra
of $Q$ over $\ofq$. The automorphism $\sz$ of $Q$ induces a Frobenius morphism
$$F_{Q,\sz;q}:\ofq Q\lra \ofq Q,\,\,\sum_{s}x_sp_s\lmto \sum_{s}x_s^q\sz(p_s),$$
 where $\sum_{s}x_sp_s$ is an $\ofq$-linear combination of
paths $p_s$ in $Q$. The set of fixed points
\begin{equation} \label{Frob-Morph-quiver-auto}
A={\frak A}(Q,\sz;q):=\{a\in A\mid F_{Q,\sz;q}(a)=a\}
\end{equation}
 is then a finitely generated $\bbF_q$-algebra. By \cite[Thm~2.1]{DD06},
$A$ is in fact isomorphic to the tensor algebra of the
$\bbF_q$-species associated with $(Q,\sz)$. Hence, it is a
hereditary $\bbF_q$-algebra.

By the definition, if $\sz=\id$, then $A=\bbF_q Q$ is
the path algebra of $Q$ over $\bbF_q$. In general, $A$ is finite dimensional
if and only if $Q$ is an acyclic quiver, i.e., a quiver containing no oriented
cycles. Moreover, every finite dimensional hereditary $\bbF_q$-algebra is Morita
equivalent to ${\frak A}(Q,\sz;q)$ for some acyclic quiver $Q$ with automorphism $\sz$.

Further, to the pair $(Q,\sz)$ we can attach a valued quiver $\ggz=\ggz(Q,\sz)=(\ggz_0,\ggz_1)$
with vertex set $I:=\ggz_0$ (resp., arrow set $\ggz_1$) the set
of $\sz$-orbits in $Q_0$ (resp., in $Q_1$). The valuation is
defined as follows: for each $i\in I=\ggz_0$, let $\vez_i$ be
the number of vertices in the orbit $i$, while for each arrow
$\rho:t\rho\longrightarrow h\rho$ in $\ggz$, let $(d_{\rho},d'_{\rho})$
be the pair of positive integers defined by
\begin{equation}\label{GCM}
d_\rho=\vez_\rho/\vez_{h\rho}\;\;
\text{and}\;\;d'_{\rho}=\vez_\rho/\vez_{t\rho},
\end{equation}
 where $\vez_\rho$ denotes the number of arrows in the orbit
$\rho$. If $\sz=\id$, the identity automorphism of $Q$, then the
valued quiver $\ggz(Q,\sz)$ coincides with $Q$.

Let $\bbZ I$ be the free abelian group with basis $I$ whose elements will be written
as either $\sum_{i\in I}x_ii$ or $(x_i)_{i\in I}$. The Euler form associated with
$\Gamma$ is defined to be the bilinear form
$\lan-,-\ran:\bbZ I \times \bbZ I\rightarrow \bbZ$ given by
$$\lan \bfx,\bfy\ran=\sum_{i\in\Gamma_0}\vez_ix_iy_i-\sum_{\rho\in\Gamma_1} \vez_\rho x_{t\rho}y_{h\rho},$$
where $\bfx=\sum x_ii,\,\bfy=\sum y_ii\in\bbZ I$.
Its symmetrization
$$(\bfx, \bfy)=\lr{\bfx, \bfy} +\lr{\bfy, \bfx}$$
 is called the symmetric Euler form.

For each vertex $i\in I$, let $e_i$ be the corresponding idempotent in $A={\frak A}(Q,\sz;q)$.
Then for each $A$-module $M$, $e_iM$ is a vector space over $D_i:=\bbF_{q^{\vez_i}}$; see
\cite[\S3.3]{DDPW}. Define
$$\udim M=\sum_{i\in I}(\dim_{D_i}e_iM)i\in\bbN I,$$
 called the dimension vector of $M$. Then for $M,N\in A\mod$,
 $$\lr{\udim M, \udim N}=\dim_{\bbF_q}\Hom_A(M, N)-\dim_{\bbF_q}\Ext^1_A(M, N).$$

\begin{Example} \label{aff-A11} Let $Q$ be the following quiver
%\begin{center}
%\setlength{\unitlength}{1mm}
%\begin{picture}(60,32)
% \put(6,15){\circle*{1}} \put(24.5,28.5){\circle*{1}}
% \put(24.5,21){\circle*{1}} \put(24.5,9){\circle*{1}}
% \put(24.5,1.5){\circle*{1}}
%
%\put(7,15.6){\vector(4,3){16}} \put(7,15.2){\vector(3,1){16}}
%\put(7,14.8){\vector(3,-1){16}} \put(7,14.4){\vector(4,-3){16}}
%
%\put(3,14.2){$0$} \put(25.6,28){$1$} \put(25.6,20.5){$2$}
% \put(25.6,9){$3$}   \put(25.6,1){$4$}
%\end{picture}
%\end{center}

\begin{center}
\begin{pspicture}(1,-0.6)(4,2.95)
 \psdot(2.2,2.4) \psdot(2.2,0)
 \psdot(1,1.2) \psdot(2.2,1.2) \psdot(3.4,1.2)
 \psline{->}(2.2,1.25)(2.2,2.4)
 \psline{<-}(2.2,0)(2.2,1.15)
 \psline{->}(2.15,1.2)(1,1.2)
 \psline{<-}(3.4,1.2)(2.25,1.2)
\uput[r](3.45,1.2){$1$} \uput[d](2.2,0){$2$}
\uput[l](1,1.2){$3$}  \uput[u](2.2,2.4){$4$}
\uput[d](2.38,1.25){$0$} \uput[r](3.1,1.9){$\sz$}
\psarc[linestyle=dashed,dash=1pt .8pt]{<-}(2.2,1.2){1.2}{2}{90}
\psarc[linestyle=dashed,dash=1pt .8pt]{<-}(2.2,1.2){1.2}{92}{180}
\psarc[linestyle=dashed,dash=1pt .8pt]{<-}(2.2,1.2){1.2}{182}{270}
\psarc[linestyle=dashed,dash=1pt .8pt]{<-}(2.2,1.2){1.2}{272}{360}
%\uput[l](3.7,-.6){{\bf Fig.~VII}\;($\widetilde D_4$)}
 \end{pspicture}
\end{center}
 with automorphism $\sz$ fixing vertex $0$ and taking
$1\mapsto 2,\,2\mapsto 3,\,3\mapsto 4,\,4\mapsto 1$ (thus, permuting $4$ arrows as well). Then
 $$A={\frak A}(Q,\sz;q)\cong \left(\begin{matrix} \bbF_q & 0 \\
                      \bbF_{q^4} & \bbF_{q^4}  \end{matrix}\right),$$
 which is a tame hereditary algebra of type $\widetilde A_{1,1}$; see \cite[Sect.~6]{DR}.
Moreover, the associated valued quiver $\ggz=\ggz(Q,\sz)$ has the form
\begin{center}
\setlength{\unitlength}{1mm}
\begin{picture}(30,10)
 \put(5,5){\circle*{1}} \put(20,5){\circle*{1}}
\put(6,5){\vector(1,0){13}}

\put(1,4){$\bar 0$} \put(21,4){$\bar 1$}   \put(12,6.4){$\rho$}
\end{picture}
\end{center}
with valuations
$$\vez_{\bar 0}=1,\, \vez_{\bar 1}=4,\;\text{ and }\;\,\vez_\rho=4.$$
%In what follows, we will simply denote this algebra by $\widetilde A_{1,1}$.
\end{Example}

For a quiver $(Q,\sz)$ with automorphism, we have the associated Hall
algebra $H(A)$ of $A={\frak A}(Q,\sz;q)$, which, by the proposition above, is generated
by its primitive elements.

Now set $v=\sqrt{q}$. In \cite{R93}, Ringel introduced a twisted multiplication
on $H(A)$ given by
$$[M]\cdot [N]=\sum_{[L]} v^{^{\lan \udim M,\udim N\ran}}F^L_{M,N}[L],\;\,\forall\, M,N\in A\mod.$$
 Then $H(A)$ with the twisted multiplication is again an associative $\bbF_q$-algebra,
which is denoted by $\cH(A)=\cH_v(Q,\sz)$ and called the Ringel--Hall algebra of $A$.
It is known from \cite{Gr95} that $\cH_v(Q,\sz)$ is indeed a (twisted) bialgebra with
comultiplication given by
$$\dt([M])=\sum_{[X],[Y]}v^{\lan \udim X,\udim Y\ran}
\frac{a_{X} a_{Y}}{a_M} F_{X,Y}^M\,  [X]\otimes [Y],\;\,\forall\, M\in A\mod.$$

  For each $i\in I$, there is a (nilpotent) simple $A$-module $S_i$ with $\udim S_i=i$.
The subalgebra of $\cH_v(Q,\sz)$ generated by the $[S_i]$ is called the composition
algebra of $A$, denoted by $\cC(A)=\cC_v(Q,\sz)$.

Clearly, $\cH_v(Q,\sz)$ is $\bbN I$-graded in terms of dimension vectors.
More explicitly,
$$\cH_v(Q,\sz)=\bigoplus_{\bfd\in\bbN I}\cH_v(Q,\sz)_{\bfd},$$
where $\cH_v(Q,\sz)_{\bfd}$ is spanned by $[M]$ with $\udim M=\bfd$.
This grading induces an $\bbN I$-grading on $\cC_v(Q,\sz)$.

On the other hand, to the valued quiver $\ggz=\ggz(Q,\sz)$ we attach a matrix
$C_{Q,\sz}=(c_{i j})_{i,j\in I}$ defined by
$$c_{ij}=\left\{ \begin{array}{ll}
2-2\displaystyle\sum_{\tau} \vez_\tau/\vez_i, \;\; &\mbox{if}\;\; i=j, \\
-\displaystyle\sum_{\rho}\vez_\rho/\vez_i, \;\; &\mbox{if}\;\; i\ne j,
                 \end{array}\right.$$
where the first sum is taken over all loops $\tau$ at $i$, and the
second is taken over all arrows $\rho$ in $\ggz$ between $i$ and
$j$. The matrix $C_{Q,\sz}$ is a (symmetrizable) Borcherds--Cartan
matrix in the sense of \cite{Bor}.

 If, moreover, $\ggz=\ggz(Q,\sz)$ contains no loops, i.e., $c_{ii}=2$ for all
$i\in I$, then $C_{Q,\sz}$ is a (symmetrizable) generalized Cartan matrix.
Let ${\frak g}_{Q,\sz}={\frak g}(C_{Q,\sz})$ be the Kac--Moody Lie algebra associated with
$C_{Q,\sz}$; see the definition in \cite{Kac90}. It is known that
$\frak g_{Q,\sz}$ admits a root space decomposition
$$\fkg_{Q,\sz}=\fkg_0\oplus\bigoplus_{\alpha\in\Phi}\fkg_\alpha,$$
 where $\Phi=\Phi(Q,\sz)$ is the root system of $\fkg_{Q,\sz}$. For the notational
simplicity, we write $i$ for the simple root $\az_i$ and, thus, view $\Phi$ as
a subset of $\bbZ I$. It is known from \cite{Kac90} that
$$\Phi(Q,\sz)=\Phi^+(Q,\sz)\cup \Phi^-(Q,\sz)$$
 with $\Phi^-(Q,\sz)=-\Phi^+(Q,\sz)$.
The roots in $\Phi^+(Q,\sz)$ (resp., $\Phi^-(Q,\sz)$) are called
positive (resp., negative) roots. Let further $\Phi_\re(Q,\sz)$ and
$\Phi_\im(Q,\sz)$ denote the set of real and imaginary roots,
respectively. Set
$$\Phi^+_\re(Q,\sz)=\Phi^+(Q,\sz)\cap\Phi_\re(Q,\sz)\;\;\text{and}\;\;
\Phi^+_\im(Q,\sz)=\Phi^+(Q,\sz)\cap\Phi_\im(Q,\sz).$$
 For each $\mu=\sum_{i\in I}\mu_i i\in\bbN I$, let $\supp\mu$ be the full
subquiver of $\ggz$ consisting of vertices $i\in I$ with $\mu_i\not=0$.
Then the fundamental set $\cF_{Q,\sz}$ is defined to be the set
$$\{\mu\in\bbN I\mid \text{$\supp\mu$ is connected and $(\mu,i)\leq 0,\;\forall\,i\in I$}\}.$$

For example, if $(Q,\sz)$ denotes the quiver with automorphism given
in Example \ref{aff-A11}. Then
$$\Phi_\im(Q,\sz)=\bbZ\dz\;\text{ with }\;\, \dz=(2,1)\;\text{ and }
\;\, \cF_{Q,\sz}=\{m\dz\mid m\geq 1\},$$
 and the associated Kac--Moody Lie algebra $\fkg_{Q,\sz}$ is the affine Lie
 algebra of type $A^{(2)}_2$ in the notation of \cite[\S4.8]{Kac90}.

%
%\begin{Rem} By results of Ringel \cite{R90a} and Green \cite{Gr95}, if
%$Q$ is an acyclic quiver, then the composition algebra $\cC_v(Q,\sz)$ is isomorphic to
%the positive part of the quantized enveloping algebra $U_v(\fkg_{Q,\sz})$.
%\end{Rem}

We now recall a characterization of primitive elements in $\cH_v(Q,\sz)$ given
in \cite{SVd2}. For each $\bfd\in\bbN I$, set
$$\cH_v(Q,\sz)^\prim_\bfd=\{x\in \cH_v(Q,\sz)_\bfd\mid \dt(x)=x\otimes 1+1\otimes x\}.$$

Following Green \cite{Gr95}, there is a symmetric bilinear form (called Green form)
$$\{-,-\}: \cH_v(Q,\sz)\times \cH_v(Q,\sz)\lra \bbC$$ defined by
$$\{[M],[N]\}=\delta_{[M],[N]}\frac{1}{a_M},\;\forall\, M,N\in A\mod,$$
 which satisfies that
$$\{xy,z\}=\{x\otimes y, \dt(z)\},\,\forall\, x,y,z\in \cH_v(Q,\sz).$$
In other words, the multiplication and comultiplication in $\cH_v(Q,\sz)$ are adjoint
with respect to Green form $\{-,-\}$.

By \cite[Lem.~3.1]{SVd1}, for each $x\in \cH_v(Q,\sz)_\bfd$,
$$\text{$x$ is primitive}\Longleftrightarrow
\{x, \Xi(Q,\sz)_\bfd\}=0,$$
 where
$$\Xi(Q,\sz)_\bfd:=\sum_{\bfd', \bfd''} \cH_v(Q,\sz)_{\bfd'} \cH_v(Q,\sz)_{\bfd''}
\subseteq\cH_v(Q,\sz)_{\bfd} $$
with the sum taking over $\bfd',\bfd''\in\bbN I$ such that
$\bfd'+\bfd''=\bfd$ and $\bfd'\not=\bfd\not=\bfd''$. Moreover,
$\cH_v(Q,\sz)^\prim_\bfd\not=0$ implies that
$$\bfd\in \{i\mid i\in I\}\cup {\cF}_{Q,\sz}.$$

%\begin{Rems} (1) By \cite[Prop.~3.2]{SVd1} (see also \cite{DX3,BS}), $\cH_v(Q,\sz)$
%is generated by primitive elements. It is also shown that $\cH_v(Q,\sz)$ is
%isomorphic to the positive part of the quantized enveloping
%algebra of a generalized Kac--Moody Lie algebra in the sense of Borcherds \cite{Bor}.
%
%(2) It has been proved in \cite{DX3} that for each $\az\in{\mathcal F}_{Q,\sz}$,
%$\dim \cH_v(Q,\sz)^\prim_\az$ is a polynomial in $q$  of rational coefficients
%which is independent of the orientation of $\ggz(Q,\sz)$, but only of the
%underlying graph of $\ggz(Q,\sz)$.
%
%(3) The primitive elements in $\cH_v(Q,\sz)$ are called cuspidal functions in
%\cite{BS} where the dimension of absolutely cuspidal functions are introduced.
%In \cite{Hen}, isotropic cuspidal functions for a quiver are studied, and those of
%tame quiver are described.
%\end{Rems}

Define a $\bbC$-linear map
\begin{equation} \label{def-int-eva}
{\mathbb I}: \cH_v(Q,\sz)\lra\bbC,\;[M]\lmto\frac{1}{a_M}.
\end{equation}
For each $\bfd\in\bbN I$, $\bbI$ restricts to a map
$\bbI_\bfd: \cH_v(Q,\sz)_\bfd\lra\bbC,\;[M]\lmto\frac{1}{a_M}$.
Further, set
$$\one_\bfd=\sum_{[M],\,\udim M=\bfd}[M]\in \cH_v(Q,\sz)_\bfd.$$
By the definition, for each $x\in \cH_v(Q,\sz)_\bfd$, we have
$$\bbI_\bfd(x)=\{x,\one_\bfd\}.$$

\begin{Prop} \label{ann-int-map} Suppose $Q$ is acyclic and take $\az=(a_i)\in\cF_{Q,\sz}$.
Then
$$\cH_v(Q,\sz)_\az^\prim\subseteq \Ker \bbI_\az.$$
\end{Prop}

\begin{proof} Since $Q$ is acyclic, $\ggz=\ggz(Q,\sz)$ is acyclic, too.
Thus, we can order all the vertices of $\ggz$ in the way
$i_1,i_2,\ldots,i_n$ ($n=|I|$) so that there are no arrows
$i_s\lra i_t$ whenever $s\geq t$. Then
$$\one_\bfd=v^{-\sum_{s<t}a_{i_s}a_{i_t}\lr{\udim S_{i_s},\udim S_{i_t}}}
[a_{i_1}S_{i_1}]\,[a_{i_2}S_{i_2}]\cdots [a_{i_n}S_{i_n}]\in \cC_v(Q,\sz)_\az.$$
Consequently, if $x\in \cH_v(Q,\sz)_\az^\prim$, then
$$\bbI_\az(x)=\{x,\one_\az\}=0.$$
\end{proof}

Now suppose $Q$ is acyclic, that is, $A={\frak A}(Q,\sz;q)$ is finite dimensional.
It is well known in this case that indecomposable
$A$-modules are divided into three classes: postprojective, regular and preinjective
ones; see \cite{R84,GR92}. An $A$-module $M$ is called popostprojective (resp., regular
or preinjective) if all its indecomposable summands are popostprojective (resp.,
regular or preinjective).

For each nonzero element $x=\sum_{[M]}\mu_M [M]\in \cH_v(Q,\sz)$,
we say that $x$ is supported on $M$ if $\mu_M\not=0$.
The following proposition is a generalization of \cite[Prop.~6.2]{Hen}
which is stated for Ringel--Hall algebras of affine quivers.

\begin{Prop} \label{prim-supp-reg} Suppose $Q$ is acyclic, and let $x$ be a primitive element
in $\cH_v(Q,\sz)_\az$, where $\az\in{\cF}_{Q,\sz}$. Then $x$ is supported on regular
$A$-modules $M$ with $\udim M=\az$.
\end{Prop}

\begin{proof} Write
$$x=\sum_{[M],\,\udim M=\az}\mu_M[M].$$
 Take $M\in A\mod $ with $\udim M=\az$ and set
$$M=P\oplus R\oplus I,$$
 where $P,R$, and $I$ are postprojective, regular and preinjective summands of $M$, respectively.
 Then
$$[M]=v^{-(\lr{\udim P,\udim R}+\lr{\udim P,\udim I}+\lr{\udim R,\udim I})} [P][R][I].$$
Suppose that $M$ is not regular, i.e., $P\not=0$ or $I\not=0$. If two of $P,R,I$ are
nonzero, then $[M]\in\Xi(Q,\sz)_\az$. Thus,
$$0=\{x,[M]\}=\frac{\mu_M }{a_M},$$
 which implies $\mu_M=0$.

It remains to consider the case where $M=P$ or $M=I$. Suppose $M=P$. Since
$\udim M=\az\in {\cF}_{Q,\sz}$, $M=P$ must be decomposable. Thus, $M$ admits a decomposition
$$M=P=t_1P_1\oplus \cdots \oplus t_sP_s,$$
 where $s,t_1,\ldots,t_s\geq 1$ with $t_1+\cdots+t_s>1$, and $P_1,\ldots,P_s$ are
 indecomposable postprojectives satisfying
$$\Hom(P_j,P_i)=0\,\;\text{ whenever }\,\; j>i.$$
 Then
$$[M]=c [P_1]^{t_1}\cdots [P_s]^{t_s}\in \Xi(Q,\sz)_\az, \;\,
\text{ for some $0\not=c\in \bbC$.}$$
 Hence, from $\{x,[M]\}=0$ it follows that
$\mu_M=0$. The case where $M=I$ can be treated similarly.
\end{proof}

\section{Primitive elements in subalgebras of $\cH_v(Q,\sz)$}

In this section we consider the subalgebra $\cH(\scrC)$ of the Ringel--Hall algebra $\cH_v(Q,\sz)$
of $A={\frak A}(Q,\sz;q)$ arising from a full subcategory $\scr C$ of $A\mbox{-mod}$ which is
closed under extensions, and relate primitive elements of $\cH_v(Q,\sz)$ with those
in $\cH(\scrC)$. We are mainly interested in the subcategory of regular $A$-modules,
as well as the subcategory filtered by an orthogonal exceptional pair in $A\mbox{-mod}$.
These two subcategories play a significant role in describing primitive
elements in Ringel--Hall algebra of tame hereditary algebras.

\medskip

As in the previous section, let $Q$ be a quiver with automorphism $\sz$ and
$A={\frak A}(Q,\sz;q)$. Let $\scrC$ be a full subcategory of $A\mod$ which
is closed under extensions. Then $\scrC$ itself is an exact category.
Let $\cH(\scrC)$ be the subspace of $\cH_v(Q,\sz)$ spanned by those
$[M]$ with $M\in\scrC$. Since $\scrC$ is closed under extensions,
$\cH(\scrC)$ becomes a subalgebra of $\cH_v(Q,\sz)$. However, in general,
$\cH(\scrC)$ is not a sub-coalgebra of $\cH_v(Q,\sz)$, but the coalgebra
structure of $\cH_v(Q,\sz)$ restricts to the one on $\cH(\scrC)$. Indeed,
for each $M\in\scrC$, set
$$\dt_{\scrC}([M])=\sum_{{[X],[Y]}\atop{X,Y\in\scrC}}v^{\lan \udim X,\udim Y\ran}
\frac{a_{X} a_{Y}}{a_M} F_{X,Y}^M\,  [X]\otimes [Y].$$
 It is easy to check that $\dt_\scrC$ is a comultiplication on $\cH(\scrC)$ with
which $\cH(\scrC)$ becomes a coalgebra. Further, $\cH(\scrC)$ inherits an $\bbN I$-grading
from $\cH_v(Q,\sz)$, i.e.,
$$ \cH(\scrC)=\bigoplus_{\bfd\in\bbN I} \cH(\scrC)_{\bfd}.$$
For each $x=\sum_{[M]}x_M[M]\in \cH_v(Q,\sz)$, write
$$x_\scrC=\sum_{[M], M\in\scrC}x_M[M]\in \cH(\scrC).$$

\medskip

\begin{Prop} \label{prim-in-subcat} Let $p\in \cH_v(Q,\sz)_\bfd$ be a primitive element. Then
$p_\scrC$ is a primitive element in $\cH(\scrC)_\bfd$.
\end{Prop}

\begin{proof} Write
$$p=\sum_{[M],\,\udim M=\bfd}\mu_M [M].$$
Then
$$p_\scrC=\sum_{{[M],\,M\in\scrC}\atop{\udim M=\bfd}}\mu_M [M]\in \cH(\scrC)_\bfd.$$
 By the definition,
$$\aligned
\dt(p)&=\sum_{[M]}\mu_M \dt([M])=\sum_{[M]}\mu_M\sum_{[X],[Y]} v^{\lan \udim X,\udim Y\ran}
\frac{a_{X} a_{Y}}{a_M} F_{X,Y}^M [X]\otimes [Y],  \\
&=p\otimes 1+1\otimes p+ \sum_{{[X],[Y]}\atop{X\not=0\not=Y}}
v^{\lan \udim X,\udim Y\ran}a_Xa_Y\big(\sum_{[M]}\frac{\mu_M}{a_M}F_{X,Y}^M \big)[X]\otimes [Y].
\endaligned$$
 Since $p$ is primitive, it follows that for all nonzero representations $X,Y$, we have
$$\sum_{[M]}\frac{\mu_M}{a_M}F_{X,Y}^M=0.$$
Moreover, if $X,Y\in\scrC$, then $M\in \scrC$ whenever $F_{X,Y}^M\not=0$ since
$\scrC$ is closed under extensions. Then
 $$\sum_{[M], M\in\scrC}\frac{\mu_M}{a_M}F_{X,Y}^M=\sum_{[M]}\frac{\mu_M}{a_M}F_{X,Y}^M=0$$
 Therefore,
$$\aligned
\dt_\scrC(p_\scrC)&=\sum_{[M],\,M\in\scrC}\mu_M \dt_\scrC([M])=\sum_{[M],\,M\in\scrC}\mu_M
\sum_{{[X],[Y]}\atop{X,Y\in\scrC}} v^{\lan \udim X,\udim Y\ran}
\frac{a_{X} a_{Y}}{a_M} F_{X,Y}^M [X]\otimes [Y]\\
&=p_\scrC\otimes 1+1\otimes p_\scrC+ \sum_{{[X],[Y]}\atop{0\not=X,Y\in\scrC}}
v^{\lan \udim X,\udim Y\ran}a_Xa_Y\big(\sum_{[M],\,M\in\scrC}\frac{\mu_M}{a_M}F_{X,Y}^M \big)[X]\otimes [Y]\\
&=p_\scrC\otimes 1+1\otimes p_\scrC,
\endaligned$$
 that is, $p_\scrC$ is a primitive element in $\cH(\scrC)$.
\end{proof}

For each $\bfd\in\bbN I$, set
$$\cH(\scrC)_{\bfd}^{\prim}=\{x\in \cH(\scrC)_{\bfd}\mid \dt_{\scrC}(x)=x\otimes 1+1\otimes x \}.$$

In the rest of this section, we assume that $Q$ is acyclic. Then
the full subcategory $\scrR_A$ of $A\mod$ consisting of regular modules is
closed under extensions; see \cite{R84}. This gives a subalgebra $\cH(\scrR_A)$
of $\cH_v(Q,\sz)$, which is also a coalgebra with comultiplication $\dt_{\scrR}$;
see \cite[Prop.~6.3]{Hen} for the affine quiver case. Moreover, $\cH(\scrR_A)$
is also $\bbN I$-graded by dimension vectors
$$\cH(\scrR_A)=\bigoplus_{\bfd\in\bbN I}\cH(\scrR_A)_{\bfd}.$$
 By Proposition \ref{prim-supp-reg}, each primitive element in $\cH_v(Q,\sz)$
is supported on regular $A$-modules. Thus, for each $\bfd\in\bbN I$, we have the inclusion
$$\cH_v(Q,\sz)_{\bfd}^{\prim}\subseteq \cH(\scrR_A)_{\bfd}^{\prim}.$$
The primitive elements in $\cH(\scrR_A)$ will be called the regular primitive elements,
which are called cuspidal regular functions in \cite{Hen}.

 Furthermore, we have the element
$$\one_\bfd^\reg=\sum_{{[M],\,M\in\scrR_A}\atop{\udim M=\bfd}}[M]\in \cH(\scrR_A)_\bfd.$$
 and the map
$$\bbI^\reg_\bfd: \cH(\scrR_A)_\bfd\lra \bbC,\; [M]\lmto \frac{1}{a_M}$$
By \cite[Thm~1(ii)]{Zh-p}, $\one_\bfd^\reg\in \cC_v(Q,\sz)_\bfd$. Hence,
if $p\in \cH_v(Q,\sz)_{\bfd}$ is primitive, then
$$\bbI^\reg_\bfd(p)=\{p,\one_\bfd^\reg\}=0.$$
As conclusion, we obtain the following result.

\begin{Prop} \label{inclusion-reg-prim} For each $\az\in\cF_{Q,\sz}$,
$$\cH_v(Q,\sz)_{\az}^{\prim}\subseteq \cH(\scrR_A)_{\az}^{\prim}\cap \Ker\bbI^\reg_\az.$$

\end{Prop}

Now we consider the full subcategory of $A\mod$ arising from an orthogonal exceptional pair.
An indecomposable $A$-module $E$ is called exceptional if $\Ext^1_A(E,E)=0$.
A pair $(E_1,E_2)$ of exceptional $A$-modules is called an orthogonal exceptional pair if
$$\Hom_A(E_1,E_2)=0=\Hom_A(E_2,E_1)\;\,\text{ and }\,\;\Ext_A^1(E_2,E_1)=0.$$
Given an orthogonal exceptional pair $(E_1,E_2)$, let ${\scr E}:=\scrE(E_1,E_2)$ denote the
full subcategory of $A\mod$ consisting of $A$-modules $M$ such that there exists an exact sequence
$$0\lra s E_2\lra M\lra t E_1\lra 0$$
 for some $s,t\geq 0$.

\begin{Prop} Let $(E_1,E_2)$ be an orthogonal exceptional pair and
$\scrE=\scrE(E_1,E_2)$. Then for each $\bfd\in\bbN I$,
$$\one_\bfd^{\scrE}\in C_v(Q,\sz)_\bfd.$$
\end{Prop}

\begin{proof} By the definition of $\scrE$,  $\one_\bfd^{\scrE}\not=0$ only if
$$\bfd=t\,\udim\, E_1+s\,\udim E_2\;\text{ for some $s,t\geq 0$}.$$
 Thus, we may suppose that $\bfd=t\,\udim E_1+s\,\udim E_2$ and write ${\bf e}=(t,s)$.

It is well known from \cite{GL, Scho} that there is an equivalence $\cF:\scrE\ra B\mod$
taking $E_i\mapsto S_i, \,i=1,2$, where $B$ is a finite dimensional hereditary
$\bbF_q$-algebra with two simples $S_1$ and $S_2$. Then $F$ induces an isomorphism
$\cH_v(B)\ra \cH(\scrE)$ of Ringel--Hall algebras. By \cite[Thm~1(ii)]{Zh-p},
$\one_{\bf e}$ lies in the composition algebra $\scrC_v(B)$ of $B$,
the subalgebra of $\cH_v(B)$ generated by $[S_1]$ and $[S_2]$.
This implies that $\one_\bfd^{\scrE}$ lies in the subalgebra of $\cH(\scrE)$ generated
by $[E_1]$ and $[E_2]$. By a result of Ringel (see \cite[Thm~2]{Zh-p}),
$[E_1], [E_2]\in \cC_v(Q,\sz)$. Therefore,
$$\one_\bfd^{\scrE}\in \cC_v(Q,\sz)_\bfd.$$
\end{proof}

\begin{Coro} \label{kernel-OEP}
Let $(E_1,E_2)$ be an orthogonal exceptional pair in $A\mod$ and
$\scrE=\scrE(E_1,E_2)$, and let $\az\in\cF_{Q,\sz}$. If $p\in \cH_v(Q,\sz)_\az$ is primitive, then
$$\{p,\one_\az^\scrE\}=\{p_\scrE,\one_\az^\scrE\}=0.$$
\end{Coro}

%\begin{Rem} By \cite{DR}, if $A$ is a tame hereditary $\bbF_q$-algebra, then
%there is an orthogonal exceptional pair $(E_1,E_2)$ in $A\mod$ such that
%$\scrE=\scrE(E_1,E_2)$ is equivalent to $\bbF_q K_2\mod$, ${\widetilde A}_{1,1}\mod$,
%or $({\widetilde A}_{1,1})^{\rm op}\mod$, where $K_2$ denotes the Kronecker
%quiver $\cdot\! \rightrightarrows\!\cdot$ and ${\widetilde A}_{1,1}={\frak A}(Q,\sz;q)$
%denotes the algebra given in Example \ref{aff-A11}. Indeed, these algebras
%are the only tame hereditary algebras with two simples (up to Morita equivalence)
%and serve as minimal models for tame hereditary algebras.
%\end{Rem}

\bigskip

\section{Primitive elements in the cyclic quiver case}

In this section we consider the Ringel--Hall algebra of a cyclic quiver
$C_r$ with $r$ vertices and recall from \cite{Sch, Hub1, Hen} a description of
primitive elements in the subalgebra spanned by nilpotent representations of $C_r$,
as well as of those in the classical Hall algebra.

For $r\geq 1$, let $C_r$ denote the cyclic quiver
%\begin{center}
%\begin{pspicture}(0,-.6)(2,2.6)
%\psset{xunit=.8cm,yunit=.7cm}
%\psdot*(1,0) \psdot*(2,0) \psdot*(3,1) \psdot*(3,2)
%\psdot*(2,3) \psdot*(1,3) \psdot*(0,2) \psdot*(0,1)
%\uput[l](0,1){$_{r-2}$} \uput[l](0,2){$_{r-1}$}
%\uput[u](1,3){$_r$} \uput[u](2,3){$_1$}
%\uput[r](3,2){$_2$} \uput[r](3,1){$_3$}
%\uput[d](2,0){$_4$} \uput[d](1,0){$_5$}
%\psline{->}(1,3)(2,3) \psline{->}(2,3)(3,2)
%\psline{->}(3,2)(3,1) \psline{->}(3,1)(2,0)
%\psline{->}(2,0)(1,0)
%\psline[linestyle=dotted,linewidth=1pt](1,0)(0,1)
%\psline{->}(0,1)(0,2) \psline{->}(0,2)(1,3)
%\end{pspicture}
%\end{center}

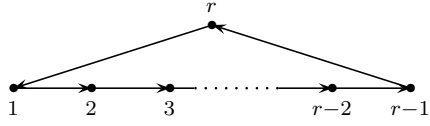
\begin{center}
\begin{pspicture}(-3,-.6)(3.6,1.6)
\psset{linewidth=0.6pt, arrowsize=3.8pt}
\psset{xunit=.8cm,yunit=.7cm}
\psdot*(-3,0) \psdot*(-1.7,0) \psdot*(-.4,0) \psdot*(2.3,0)
\psdot*(3.6,0)\psdot*(.3,1.2) \uput[u](.3,1.2){$_{r}$}
\uput[d](-3,0){$_1$} \uput[d](-1.7,0){$_2$} \uput[d](-.4,0){$_3$}
\uput[d](2.3,0){$_{r-2}$} \uput[d](3.6,0){$_{r-1}$}
\psline{->}(-3,0)(-1.7,0) \psline{->}(-1.7,0)(-.4,0)
\psline(-.4,0)(0,0)
\psline[linestyle=dotted,linewidth=1pt](0,0)(1.4,0)\psline{->}(1.4,0)(2.3,0)
\psline{->}(2.3,0)(3.6,0) \psline{->}(3.6,0)(.3,1.2)
\psline{->}(.3,1.2)(-3,0)
\end{pspicture}
\end{center}
with vertex set $I=\bbZ/r\bbZ=\{1,2,\ldots,r\}$ and arrow set
$\{i\to i+1\mid i\in I\}$. By convention, $C_1$ is the Jordan quiver with
only one vertex and one loop. For each $r>1$,
$$\Phi_\im^+(C_r)=\bbN\dz\;\,\text{ with }\;\, \dz=(1,\ldots,1)\in\bbN I.$$

By $\rep C_r=\rep_{\bbF_q} C_r$ we denote the category of finite dimensional
representations of $C_r$ over $\bbF_q$, which are identified with finite
dimensional left modules over the path algebra $\bbF_qC_r$.

Let $\rep^0 C_r$ denote the full subcategory of $\rep C_r$ consisting
of nilpotent representations of $C_r$. For each vertex $i\in I$, we have
a one-dimensional nilpotent representation $S_i$. The $S_i$ form a complete set
of simple objects in $\rep^0 C_r$. Further, for each integer $l\geq 1$,
there is a unique (up to isomorphism) indecomposable
representation $S_i[l]$ in $\rep^0 C_r$ of length $l$ with top $S_i$. It
is well known that $S_i[l]$, $i\in I, l\geq 1$, yield all
isoclasses of indecomposable representations in $\rep^0 C_r$, and that
the Auslander--Reiten quiver of $\rep^0 C_r$ is a tube of rank $r$.

To $C_r$ we have the associated Ringel--Hall algebra $\cH_v(C_r)$ of $C_r$,
as well as its subalgebra $\cH^0_v(C_r)$ spanned by $[V]$ with $V\in \rep^0 C_r$.
Both of them are bialgebras. By \cite{Sch,Hub1}, $\cH^0_v(C_r)$ is generated by
its primitive elements.

In case $r=1$, $\cH^0_v(C_1)$ is the classical Hall algebra studied in \cite{St,Ha,Mac,Zel},
which is $\bbN$-graded by dimensions. For each partition $\lz=(\lz_1,\ldots,\lz_m)$
with $\lz_1\geq \cdots\geq\lz_m\geq 1$, set $\ell(\lz)=m$, called the length of $\lz$.
Further, set
$$I_\lz=\bigoplus_{i=1}^m S_1[\lz_i]\in \rep^0 C_1.$$
Then, by writing $\lz$ in the ``exponential form'' as $\lz=(1^{t_1},2^{t_2},\ldots)$,
we have
\begin{equation} \label{auto-group-partition}
a_\lz(q):=|{\rm Aut}(I_\lz)|=q^{|\lz|+2n(\lz)}\prod_{i\geq 1}\prod_{j=1}^{t_i}(1-q^{-j}),
\end{equation}
 where $|\lz|=\lz_1+\lz_2+\cdots+\lz_m$ and $n(\lz)=\sum_{i=1}^{\ell(\lz)}(i-1)\lz_i$;
 see \cite[Ch.~II~(1.6)]{Mac}.

Moreover, for each $n\geq 1$, $\dim \cH^0_v(C_1)^\prim_n =1$, and
\begin{equation} \label{prim-elt-Jordan}
p_n=\sum_{\lz \vdash n}\bigg(\prod_{s=1}^{\ell(\lz)-1}(1-q^s)\bigg)[I_\lz]\in \cH^0_v(C_1)^\prim_n;
\end{equation}
see \cite[Ch.~III, Sect.~7]{Mac} or \cite[\S5]{Hub1}. In other words,
$\cH^0_v(C_1)$ is generated by the primitive elements $p_n$.

\medskip

Now suppose $r>1$. Following Schiffmann \cite{Sch} and Hubery \cite{Hub1},
$\cH^0_v(C_r)$ admits a decomposition
$$\cH^0_v(C_r)=\cC_v(C_r)\otimes \bbC[c_1,c_2,\ldots],$$
 where $\cC_v(C_r)$ is the composition algebra of $C_r$ (i.e., the subalgebra
of $\cH^0_v(C_r)$ generated by the $[S_i]$), and $\bbC[c_1,c_2,\ldots]$
is an infinite polynomial ring with
$$c_n=(-1)^nv^{-2rn}\sum_{{[M]: \udim M=n\dz}\atop{{\rm soc} M}\;\text{square free}}
(-1)^{{\rm dim}\,{\rm End}(M)}a_M [M]$$
central in $\cH^0_v(C_r)$. Also, it is proved in \cite[Prop.~9]{Hub1} that
for each $n\geq 1$,
$$\dt(c_n)=\sum_{s=0}^n c_s\otimes c_{n-s},$$
 where $c_0=1$. Further, define inductively
$$x_n=nc_n-\sum_{s=1}^{n-1} x_s c_{n-s}.$$
Then, by \cite[Cor.~10]{Hub1}, $x_n$ are primitive elements in $\cH^0_v(C_r)$.
This together with \cite[Thm~5.4]{Hen} implies that
$$\cH^0_v(C_r)^{\prim}_{n\dz}=\bbC x_n,\;\;\forall\, n\geq 1.$$
Then, by \cite[Lem.~12]{Hub2} (see also \cite[Lem.~2.2.1]{DDF} and \cite[Lem.~5.5]{Hen}),
we have
$$x_n=v^{n-2rn}(v^n-v^{-n})\sum_{i=1}^r [S_i[rn]]+y_n,$$
 where $y_n$ is a linear combination of $[M]$ with $M$ decomposable and
 $\udim M=n\dz$. Finally, set
\begin{equation} \label{prim-elt-cyc-r}
p_n^{(r)}=\frac{v^{2rn-n}}{v^n-v^{-n}}x_n=\sum_{i=1}^r [S_i[rn]]+\frac{v^{2rn-n}}{v^n-v^{-n}}y_n,
\;\,\forall\, n\geq 1,
\end{equation}
which are called the normalized primitive elements in \cite[6.3.1]{Hen}.
If $r=1$, then $p_n^{(1)}$ coincides with $p_n$ given in \eqref{prim-elt-Jordan}.
However, as pointed out in \cite[Rem.~5.6]{Hen}, there  is no known closed
formula for $p_n^{(r)}$ in general.

For each $1\leq t<r$, there is an embedding $\Upsilon_{r,t}: \rep^0 C_t\ra \rep^0 C_r$
taking
$$S_i\lmto \begin{cases} S_i, & \text{if $1\leq i<t$},\\
                S_t[r-t+1], & \text{if $i=t$};\end{cases}$$
 see, e.g., \cite[\S4]{Hub1}. This then gives rise to an embedding of algebras
$$\zeta_{r,t}:\cH^0_v(C_t)\lra\cH^0_v(C_r).$$
Thus, if we denote by $\cA_t$ the full subcategory of $\rep^0 C_r$
consisting of $\Upsilon_{r,t}(V)$ with $V\in\rep^0 C_t$, then $\cA_t$ is equivalent
to $\rep^0 C_t$ and the subalgebra $\cH(\cA_t)$ of $\cH^0_v(C_r)$ is isomorphic to
$\cH^0_v(C_t)$. As a consequence of Proposition \ref{prim-in-subcat}, we obtain
the following result.

\begin{Prop} For each $1\leq t<r$ and $n\geq 1$, $(p_n^{(r)})_{\cA_t}$ is a primitive element
in $\cH(\cA_t)$. In other words,
$$p_n^{(r)}=\zeta_{r,t}(p_n^{(t)})+\text{other terms}.$$
Here ``other terms" means a combination of $[V]$ with $V\not\in\cA_t$.
\end{Prop}

We remark that there are several ways to embed $\rep^0 C_t$ into $\rep^0 C_r$.
For example, for each sequence ${\bf r}=(r_1,\ldots,r_t)$ of positive
integers with $r=r_1+\cdots +r_t$ and a fixed vertex $1\leq j\leq r$, there is an embedding
$\Upsilon_{\bf r}: \rep^0 C_t\ra \rep^0 C_r$ taking
$$S_i\lmto S_{j+r_1+\cdots+r_{i-1}}[r_i],\;\,\forall\,1\leq i\leq t.$$
Then with respect to the embedding $\Upsilon_{\bf r}$, an analogous result of the proposition
above holds. Hence, this provides an inductive way to understand primitive
elements $p_n^{(r)}$ in $\cH(\cA_r)$. In particular, take $t=1$ and $j=1$.
This gives an embedding
$$\Upsilon_{r,1}: \rep^0 C_1\lra \rep^0 C_r,\;\, S_1\lmto S_1[r].$$
For each partition $\lz=(\lz_1,\ldots,\lz_m)\vdash n$, set
$$I_\lz^{(r)}=\bigoplus_{i=1}^m S_1[\lz_i r]\in \rep^0 C_r.$$
  Then
\begin{equation} \label{prim-elt-cyclic-decomp}
p_n^{(r)}=\sum_{\lz \vdash n}\bigg(\prod_{s=1}^{\ell(\lz)-1}(1-q^s)\bigg)[I^{(r)}_\lz]
+\text{other terms}.
\end{equation}

As a conclusion of discussions above, we obtain the following description
of primitive elements in $\cH_v(C_r)$. Let $\bbF_q[T]$
denote the polynomial ring in one indeterminate $T$ over $\bbF_q$ and set
$$\bbA_q^1:=\bbA_{\bbF_q}^1=\spec (\bbF_q[T]).$$
It is well known that $\rep_{\bbF_q} C_r$ admits a decomposition
$$\rep_{\bbF_q} C_r=\rep^0_{\bbF_q} C_r\times \prod_{x\in \bbA_q^1\backslash\{0\}} \cT_x,$$
 where for each $x\in \bbA_q^1\backslash\{0\}$, $\cT_x$ is equivalent to $\rep^0_{\bbF_{q^d}} C_1$
with $d=\deg(x)$. In particular, for such an $x$ with $\deg(x)=d$, $\cT_x$ contains a
unique (up to isomorphism) simple object (called quasi-simple) $E_x$ of dimension vector
$d\dz$ (with $\End(E_x)\cong\bbF_{q^d}$) and indecomposable representations $E_x[n]$
of quasi-length $n$ for all $n\geq 1$. For each partition $\lz=(\lz_1,\ldots,\lz_m)$, set
$$I_\lz(x)=\bigoplus_{i=1}^m E_x[\lz_i]\in \add\cT_x.$$
Then by modifying the primitive element $p_n$ given in \eqref{prim-elt-Jordan},
we obtain a primitive element
$$p_n(x)=\sum_{\lz \vdash n}\bigg(\prod_{s=1}^{\ell(\lz)-1}(1-q^{ds})\bigg)[I_\lz(x)]$$
in $\cH_v(C_r)$ of dimension vector $dn\dz$. Consequently, we obtain the following
result; see \cite[Prop.~5.7]{Hen}.

\begin{Prop} \label{prim-elt-cyclic} For each $n\geq 1$, the set
$$\{p_n^{(r)}\}\cup \{p_m(x)\mid 0\not=x\in \bbA_q^1\;\text{ with }\; m\deg(x)=n\}$$
 forms a basis of $\cH_v(C_r)^\prim_{n\dz}$.
\end{Prop}

\section{Regular primitive elements for tame hereditary algebras}

In this section we consider the algebra $A={\frak A}(Q,\sz;q)$ over
$\bbF_q$ associated with an acyclic tame quiver $Q$ with automorphism $\sz$, and
make a comparison between the dimension of the space of
primitive elements in $\cH_v(Q,\sz)$ and that of regular primitive
elements with a fixed dimension vector.

\medskip

Throughout this section, $Q=(Q_0,Q_1)$ always denotes an acyclic tame quiver
with automorphism $\sz$. Then $A={\frak A}(Q,\sz;q)$ is a tame hereditary
$\bbF_q$-algebra, and
$$\cF_{Q,\sz}=\{n\dz\mid n\geq 1\},$$
 where $\dz$ denotes the minimal positive imaginary root of ${\frak g}_{Q,\sz}$.

As in Section 3, let $\scrR_A$ be the full subcategory of $A\mod$
consisting of regular $A$-modules. It is well known from \cite{DR,R76}
(see also \cite{R84,GR92}) that the Auslander--Reiten quiver of $\scrR_A$
consists of infinitely many tubes. Moreover, if $\sz=\id$, then all the tubes in
$\scrR_A$ are indexed by the projective line $\bbP_{\bbF_q}^1$ of $\bbF_q$,
that is, the set of irreducible homogeneous polynomials in $\bbF_q[T_0,T_1]$
of two variables $T_0$ and $T_1$. For this reason, we use $\bbP_{\bbF_q,\sz}^1$
to index the set of all tubes in $\scrR_A$ for $A={\frak A}(Q,\sz;q)$ with
an arbitrary automorphism $\sz$. For each $x\in \bbP_{\bbF_q,\sz}^1$,
let $\cT_x$ denote the corresponding tube in $\scrR_A$ and let $r_x$ be its rank.
Then almost all tubes in $\scrR_A$ are homogeneous, i.e., those $\cT_x$ with $r_x=1$,
and there are at most $3$ non-homogeneous tubes. Moreover, each homogeneous
tube $\cT_x$ admits a unique simple object $E_x$ (i.e., the quasi-simple module
in $\cT_x$).

\begin{Rem} \label{homog-tube-sigma}
In fact, every tube in $A\mod$ can also be obtained by folding those
in $\ofq Q\mod$ via the Frobenius twist functor $\ofq Q\mod\ra \ofq Q\mod$
induced by $\sz$; see \cite[\S2.2]{DDPW} and \cite{DOP}. More precisely,
folding homogeneous tubes in $\ofq Q\mod$ always gives rise to homogeneous ones in
$A\mod$, while folding non-homogeneous tubes in $\ofq Q\mod$ give either
non-homogeneous or homogenous ones in $A\mod$. Furthermore, if a homogeneous tube $\cT_x$
in $A\mod$ is obtained by folding homogeneous tubes of $F$-period $p$ in
$\ofq Q\mod$, then its quasi-simple module $E_x$ satisfies
$$\udim E_x=p\dz\;\text{ and }\;\;\End(E_x)\cong\bbF_{q^p}.$$
If a homogeneous tube $\cT_x$ in $A\mod$ is obtained by folding
non-homogeneous tubes in $\ofq Q\mod$, then there are two positive integers
$d_x$ and $e_x$ with $d_x\mid e_x$ such that
$$\udim E_x=d_x\dz\;\text{ and }\;\, \End(E_x)\cong \bbF_{q^{e_x}}.$$
We remark that if $\sz=\id$, then each non-homogeneous tube in $\ofq Q\mod$
gives a non-homogeneous tube with the same rank in $A\mod$, and, moreover,
$d_x=e_x$ for all homogeneous tubes $\cT_x$.
\end{Rem}

\begin{Examples} (1) Let $(Q,\sz)$ be the quiver with automorphism given in
Example \ref{aff-A11} with $A={\frak A}(Q,\sz;q)$. Then there are three
tubes of rank $2$ in $\ofq Q\mod$, one of them is folded into a homogeneous
tube in $A\mod$, say $\cT_{x_0}$, while the other two are folded
into a homogeneous tube, say $\cT_{x_1}$. Moreover, their quasi-simple modules
$E_{x_0}$ and $E_{x_1}$ satisfy
$$\udim E_{x_0}=\dz,\;\End(E_{x_0})\cong \bbF_{q^2},\;
\udim E_{x_1}=2\dz,\;\text{ and }\;\, \End(E_{x_1})\cong \bbF_{q^4}.$$

(2) Let $Q$ be the quiver with automorphism $\sz$ which has $n=2 \ell$ vertices
($\ell\geq 2$):
 
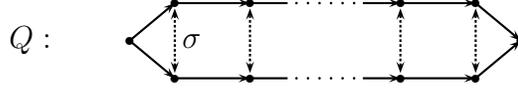
\begin{center}
\begin{pspicture}(-3.4,-.8)(10,0.9)
\put(-1.2,-0.1){$Q:$}
\psdot*[dotsize=3pt](.4,0) \psdot*[dotsize=3pt](1,.5)
\psdot*[dotsize=3pt](1,-.5) \psdot*[dotsize=3pt](2,.5)
\psdot*[dotsize=3pt](2,-.5) \psdot*[dotsize=3pt](4,0.5)
\psdot*[dotsize=3pt](4,-0.5) \psdot*[dotsize=3pt](5,0.5)
\psdot*[dotsize=3pt](5,-0.5) \psdot*[dotsize=3pt](5.6,0)
 \psline{->}(.4,0)(1,.5)
\psline{->}(.4,0)(1,-.5) \psline{->}(1,.5)(2,.5)
\psline{->}(1,-.5)(2,-.5) \psline{-}(2,.5)(2.5,.5)
\psline{-}(2,-.5)(2.5,-.5)
\put(1.1,-0.1){$\sz$}
\psline[linestyle=dotted,linewidth=1pt](2.6,.5)(3.4,.5)
\psline[linestyle=dotted,linewidth=1pt](2.6,-.5)(3.4,-.5)
\psline{->}(3.5,.5)(4,.5) \psline{->}(3.5,-.5)(4,-.5)
\psline{->}(4,.5)(5,.5) \psline{->}(4,-.5)(5,-.5)
\psline{->}(5,.5)(5.6,0) \psline{->}(5,-.5)(5.6,0)
\psline[linestyle=dashed,dash=1pt .8pt]{<->}(1,0.43)(1,-0.43)
\psline[linestyle=dashed,dash=1pt .8pt]{<->}(2,0.43)(2,-0.43)
\psline[linestyle=dashed,dash=1pt .8pt]{<->}(4,0.43)(4,-0.43)
\psline[linestyle=dashed,dash=1pt .8pt]{<->}(5,0.43)(5,-0.43)
%\uput[d](3.2,-.9){{\bf FIG.~1}\;($\widetilde A_{2\ell-1}$)}
\end{pspicture}
\end{center}
Then the underlying graph of $\ggz(Q,\sz)$ is of type ${\widetilde B}_\ell$.
It is known that there are two non-homogeneous tubes of rank $\ell$ in $\ofq Q\mod$,
which are folded into a unique non-homogeneous tube $\cT_{x_0}$ of rank $\ell$ in
$A\mod$ such that the sum of dimension vectors of non-isomorphic
quasi-simple modules in $\cT_{x_0}$ equals $2\dz$.

\end{Examples}

Each tube $\cT_x$ in $A\mod$ gives
an exact subcategory $\add\cT_x$ of $\scrR_A$ which defines a subalgebra
$\cH(\cT_x):=\cH(\add\cT_x)$ of $\cH(\scrR_A)$. Since there are no nonzero
morphisms and no nontrivial extensions between modules in distinct tubes, all the subalgebras
$\cH(\cT_x)$ are closed under the comultiplication $\dt_\scrR$ in $\cH(\scrR_A)$,
and thus, each $\cH(\cT_x)$ is a sub-bialgebra of $\cH(\scrR_A)$. Therefore,
$$\cH(\scrR_A)^{\prim}_{n\dz}=\bigoplus_{x\in \bbP_{\bbF_q,\sz}^1}
\cH(\cT_x)^{\prim}_{n\dz},\;\,\forall\, n\geq 1.$$

For each $x\in \bbP_{\bbF_q,\sz}^1$ with $r_x>1$ (i.e., $\cT_x$ is a
non-homogeneous tube), the sum of dimension vectors of $r_x$ non-isomorphic
quasi-simple modules in $\cT_x$ equals $d_x\dz$ for some positive
integer $d_x\geq 1$ (note that $d_x=1$ if $\sz=\id$), and there is an
equivalence
$$\Rep^0_{\bbF_{q^{d_x}}}C_{r_x}\cong \add\cT_x,$$
which induces an isomorphism of bialgebras
$$\cH_{v^{d_x}}^0(C_{r_x})\cong \cH(\cT_x).$$
Then for each $m\geq 1$, the primitive element $p_m^{(r_x)}$ in
$\cH_{v^{d_x}}^0(C_{r_x})$ defined in \eqref{prim-elt-cyc-r} gives rise to a
(normalized) primitive element $p_m(x)$ in $\cH(\cT_x)_{md_x\dz}$ (by replacing
$q$ by $q^{d_x}$) which has
the form
$$p_m(x)=\sum_{i=1}^{r_x} [M_i]+\text{other terms},$$
where $M_1,\ldots,M_{r_x}$ are all pairwise non-isomorphic indecomposable
modules in $\cT_x$ with dimension vector $md_x\dz$.

Now let $\cT_x$ be a homogeneous tube in $A\mod$ with
$$\udim E_x=d_x\dz\;\text{ and }\;\, \End(E_x)\cong \bbF_{q^{e_x}},$$
 where $d_x,e_x\geq 1$ satisfy $d_x\mid e_x$. Note that $d_x=e_x$ unless
 $\cT_x$ is obtained by folding non-homogeneous tubes in $\ofq Q\mod$;
see the remark above. In this case, we have an equivalence
$$\Rep^0_{\bbF_{q^{e_x}}}C_1\cong \add\cT_x,$$
as well as a bialgebra isomorphism
$$\cH_{v^{e_x}}^0(C_1)\cong \cH(\cT_x).$$
 Thus, for each $m\geq 1$, the primitive element $p_m$ in
$\cH_{v^{e_x}}^0(C_1)$ defined in \eqref{prim-elt-Jordan}
 gives a primitive element
$$p_m(x)=\sum_{\lz \vdash m}\prod_{s=1}^{\ell(\lz)-1}(1-q^{se_x})
[I_\lz(x)]\in \cH(\cT_x)_{md_x\dz}^\prim,$$
where for a partition $\lz=(\lz_1,\lz_2,\ldots,\lz_t) \vdash m$ with
$t=\ell(\lz)$, $I_\lz(x)=\oplus_{i=1}^t E_x[\lz_i]$ with $E_x[\lz_i]$
the indecomposable module in $\cT_x$ of quasi-length $\lz_i$ (or, equivalently,
of dimension vector $\lz_id_x\dz$).

For notational convenience, for each $x\in \bbP_{\bbF_q,\sz}^1$, we call
$d_x$ the degree of $x$, denoted by $\deg(x)=d_x$. If $\cT_x$ is non-homogeneous,
we set $e_x=d_x$.

As an analogy of Proposition \ref{prim-elt-cyclic}, we obtain the following result;
see also \cite[Sect.~6.3]{Hen}.

\begin{Prop} \label{prim-elt-reg-tame} For each $n\geq 1$, the set
$$\{p_m(x)\mid x\in \bbP_{\bbF_q,\sz}^1\;\text{ with }\; m \deg(x)=n\}$$
 forms a basis of $\cH_v(\scrR_A)^\prim_{n\dz}$. In particular, we have
$$\dim \cH(\scrR_A)_{n\dz}^\prim=\sharp\{x\in \bbP_{\bbF_q,\sz}^1\mid \deg(x)|n\}.$$
\end{Prop}

For each $n\geq 1$, let $I_{Q,\sz}(n\dz,q)$ denote the number of isoclasses of
indecomposable $A$-modules of dimension vector $n\dz$, and let $\mult\,n\dz$ be
the multiplicity of $n\dz$ in $\fkg_{Q,\sz}$, i.e.,
$\mult\,n\dz=\dim ({\fkg}_{Q,\sz})_{n\dz}$.

The following result is given in \cite{HX,ZZG} for $\sz=\id$, and in \cite{Obul, DOP}
for the remaining cases.

\begin{Lem} \label{codim=1} For each $n\geq 1$,
$$\dim\cH_v(Q,\sz)_{n\dz}^\prim=I_{Q,\sz}(n\dz,q)-{\rm mult}\,n\dz
=\sum_{s\mid n}\vphi_s(q),$$
where $\vphi_s(q)$ denotes the number of monic irreducible polynomials
of degree $s$ over $\bbF_q$.
\end{Lem}

Further, on the one hand, we have
$$I_{Q,\sz}(n\dz,q)=\sum_{x\in \bbP_{\bbF_q,\,\sz}^1,\,\deg(x)|n}r_x.$$
On the other hand, a case-by-case check shows that
$$\mult\,n\dz=1+\sum_{x\in \bbP_{\bbF_q,\,\sz}^1,\,\deg(x)|n}(r_x-1).$$
We conclude from the equalities above that
$$\aligned
\dim\cH_v(Q,\sz)_{n\dz}^\prim&=I_{Q,\sz}(n\dz,q)-{\rm mult}\,n\dz\\
&=I_{Q,\sz}(n\dz,q)-\big(1+\sum_{x\in \bbP_{\bbF_q,\,\sz}^1,\,\deg(x)|n}(r_x-1)\big)\\
&=\sharp\{x\in \bbP_{\bbF_q,\sz}^1\mid \deg(x)|n\}-1\\
&=\dim \cH(\scrR_A)_{n\dz}^\prim-1,
\endaligned$$
 that is, we obtain the following result; see \cite[6.4]{Hen} for
the quiver case.

\begin{Prop} \label{codim=1} For each $n\geq 1$,
$$\dim\cH(\scrR_A)_{n\dz}^\prim=\dim \cH_v(Q,\sz)_{n\dz}^\prim+1.$$
\end{Prop}

\bigskip

\section{The proof of Theorem \ref{Main-thm-1}}

This section is mainly devoted to the proof of Theorem \ref{Main-thm-1}. We keep all the
notations in previous sections which will be used without explanation throughout.

\medskip

The following result can be deduced from \cite[Thm~6.9]{Hen}.
However, we provide a direct proof here.

\begin{Lem} \label{key-lemma} Let $C_r$ be the cyclic quiver with $r$ vertices
and for each $n\geq 1$, set
$$\one_{n\dz}^{(r)}=\sum_{{[M],\,M\in \Rep^0_{\bbF_q}C_r}\atop{\udim M=n\dz}}[M]\in \cH^0(C_r).$$
Then
$$\big\{p_n^{(r)}, \one_{n\dz}^{(r)}\big\}
=\sum_{\lz \vdash n}\frac{\prod_{s=1}^{\ell(\lz)-1}(1-q^s)}{a_\lz(q)}.$$
\end{Lem}

\medskip

\begin{proof} If $r=1$, this follows directly from the definition
of $p_n^{(1)}$ and $\one_{n\dz}^{(1)}$.

Now suppose $r>1$ and consider the following quiver $\widehat C_{r+1}$
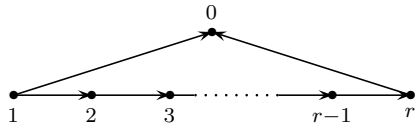
\begin{center}
\begin{pspicture}(-3,-.6)(3.6,1.6)
\psset{linewidth=0.6pt, arrowsize=3.8pt}
\psset{xunit=.8cm,yunit=.7cm}
\psdot*(-3,0) \psdot*(-1.7,0) \psdot*(-.4,0) \psdot*(2.3,0)
\psdot*(3.6,0)\psdot*(.3,1.2) \uput[u](.3,1.2){$_{0}$}
\uput[d](-3,0){$_1$} \uput[d](-1.7,0){$_2$} \uput[d](-.4,0){$_3$}
\uput[d](2.3,0){$_{r-1}$} \uput[d](3.6,0){$_{r}$}
\psline{->}(-3,0)(-1.7,0) \psline{->}(-1.7,0)(-.4,0)
\psline(-.4,0)(0,0)
\psline[linestyle=dotted,linewidth=1pt](0,0)(1.4,0)\psline{->}(1.4,0)(2.3,0)
\psline{->}(2.3,0)(3.6,0) \psline{->}(3.6,0)(.3,1.2)
\psline{->}(-3,0)(.3,1.2)
\end{pspicture}
\end{center}
with vertex set $\wh I:=\{0,1,\ldots,r\}$. Then $A=\bbF_q\wh C_{r+1}$ is a finite
dimensional tame hereditary algebra. To avoid confusion, for each $i\in\wh I$,
let $\wh S_i$ denote the corresponding simple $A$-module and let $\wh\dz$ be the
associated minimal positive imaginary root.

On the one hand, the full subcategory $\scrR_A$ consisting of all regular $A$-modules
admits infinitely many tubes $\cT_x$, $x\in \bbP_{\bbF_q}^1$, one of which is
a tube of rank $r$, denoted by $\cT_\infty$, and all the others are homogeneous.
By Proposition \ref{prim-elt-reg-tame}, for each $n\geq 1$, we obtain a basis
$$\{p_m(x)\mid x\in \bbP_{\bbF_q}^1\;\text{ with }\; m \deg(x)=n\}$$
 of $\cH_v(\scrR_A)^\prim_{n\wh\dz}$, where $p_m(x)$ is the normalized
primitive element in $\cH(\cT_x)$. Further, by Proposition \ref{inclusion-reg-prim},
the linear function
$$\bbI^\reg_{n\wh\dz}: \cH(\scrR_A)^{\prim}_{n\dz}\lra \bbC,\;
z\lmto \{z,\one_{n\wh\dz}^\reg\}$$
 satisfies that
 $$\cH(A)^{\prim}_{n\dz}\subseteq \Ker \bbI^\reg_{n\wh\dz}.$$
Since for each $n\geq 1$, there exists $x\in \bbP_{\bbF_q}^1$ with $\deg(x)=n$
such that $\cT_x$ is a homogeneous tube, it follows that
$$\bbI^\reg_{n\wh\dz}(p_1(x))=\{p_1(x), \one_{n\wh\dz}^\reg\}=\frac{1}{q^n-1}\not=0.$$
 Hence, $\bbI^\reg_{n\wh\dz}\not=0$. This together with
Proposition \ref{codim=1} implies that
 $$\cH(A)^{\prim}_{n\dz}=\Ker \bbI^\reg_{n\wh\dz},$$
that is, Theorem \ref{Main-thm-1} holds for the quiver $\wh C_{r+1}$.

On the other hand, let $E_1$ be the indecomposable representation of
$\wh C_{r+1}$ with $\udim E_1=\wh\dz-\udim \wh S_0$, and set $E_2=\wh S_0$.
Then $(E_1,E_2)$ is an orthogonal exceptional pair in $A\mod$ such that the full
subcategory $\scrE=\scrE(E_1,E_2)$ contains all homogeneous tubes $\cT_x$, and that
$\scrE\cap\cT_\infty$ becomes a homogeneous tube whose quasi-simple
module $E_\infty$ is the indecomposable representation of $\wh C_{r+1}$ with
$\udim E_\infty=\wh\dz$ and ${\rm soc} E_\infty=\wh S_0\oplus \wh S_r$.
As stated in Section 3, $\scrE$ gives rise to a linear function
$$\bbI^\scrE_{n\wh\dz}: \cH(\scrR_A)^{\prim}_{n\dz}\lra \bbC,\;
z\lmto \{z,\one_{n\wh\dz}^\scrE\},$$
 which is nonzero by a similar argument as in the proof of
$\bbI^\reg_{n\wh\dz}\not=0$. Combining this with Corollary \ref{kernel-OEP} shows that
$$\cH(A)^{\prim}_{n\dz}=\Ker \bbI^\scrE_{n\wh\dz}.$$

 Since for each homogeneous tube $\cT_x$ with $\deg(x)=1$, we have
$$\bbI^\reg_{n\wh\dz}(p_n(x))=\bbI^\scrE_{n\wh\dz}(p_n(x)),$$
 we conclude that $\bbI^\reg_{n\wh\dz}=\bbI^\scrE_{n\wh\dz}$ as linear functions.
This together with \eqref{prim-elt-cyclic-decomp} implies particularly that
$$\{p_n(\infty),\one_{n\wh\dz}^\reg\}=\{p_n(\infty),\one_{n\wh\dz}^\scrE\}
=\{p_n(\infty)_{\scrE},\one_{n\wh\dz}^\scrE\}
=\sum_{\lz \vdash n}\frac{\prod_{s=1}^{\ell(\lz)-1}(1-q^s)}{a_\lz(q)}.$$

 Clearly, there is an embedding from $\rep_{\bbF_q}^0 C_r$ into $A\mod$
 via taking
$$S_i\lmto \wh S_i,\;\forall\, 2\leq i\leq r,\;S_1\lmto \wh S_{1,0},$$
 where $\wh S_{1,0}$ is the indecomposable $A$-module of dimension $2$ with
 top $\wh S_1$ and socle $\wh S_0$. Moreover, via this embedding,
$\rep_{\bbF_q}^0 C_r$ is identified with the full subcategory $\add \cT_\infty$ of $A\mod$.
Thus, we have an algebra isomorphism
$$\cH^0(C_r)\stackrel{\cong}{\lra}\cH(\cT_\infty)$$
 which takes
$$p_n^{(r)}\lmto p_n(\infty),\;\;\one_{n\dz}^{(r)}\lmto \one_{n\wh\dz}^{{\rm add}\cT_\infty}.$$
 Therefore,
$$\big\{p_n^{(r)}, \one_{n\dz}^{(r)}\big\}=\{p_n(\infty),\one_{n\wh\dz}^{{\rm add}\cT_\infty}\}
=\{p_n(\infty),\one_{n\wh\dz}^\reg\}
=\sum_{\lz \vdash n}\frac{\prod_{s=1}^{\ell(\lz)-1}(1-q^s)}{a_\lz(q)}.$$
 This finishes the proof of the lemma.
\end{proof}

\noindent{\bf Proof of Theorem \ref{Main-thm-1}}. Let $Q$ be an arbitrary tame acyclic
quiver and $\sz$ be an automorphism of $Q$. Then $A={\frak A}(Q,\sz;q)$ is a tame
hereditary algebra over $\bbF_q$, and for each fixed positive integer $n$,
by Proposition \ref{prim-elt-reg-tame}, we have a basis
$$\{p_m(x)\mid x\in \bbP_{\bbF_q,\sz}^1\;\text{ with }\; m \deg(x)=n\}$$
of $\cH_v(\scrR_A)^\prim_{n\dz}$. Since $\cH(\cT_x)\cong \cH_{v^{e_x}}^0(C_{r_x})$
for each $x\in \bbP_{\bbF_q,\sz}^1$, where $r_x$ is the rank of $\cT_x$, it
follows from Lemma \ref{key-lemma} that if $m\deg(x)=n$, then
$$\{p_m(x),\one_{n\dz}^\reg\}=\sum_{\lz \vdash m}
\frac{\prod_{s=1}^{\ell(\lz)-1}(1-q^{se_x})}{a_\lz(q^{e_x})}.$$
 By Propositions \ref{inclusion-reg-prim} and \ref{codim=1}, it remains to
 show that the linear function
$$\bbI^\reg_{n\dz}: \cH(\scrR_A)^{\prim}_{n\dz}\lra \bbC$$
is nonzero. Since there always exists a tube $\cT_x$ with $\deg(x)=n$,
we have $p_1(x)\in \cH(\scrR_A)^{\prim}_{n\dz}$ satisfying
$$\{p_1(x), \one_{n\dz}^\reg\}=\frac{1}{q^{e_x}-1}\not=0.$$
Therefore, $\bbI^\reg_{n\dz}$ is nonzero. This finishes the proof.

\medskip

To prove Theorem \ref{Main-thm-2}, we need the following lemma. Its proof involves
Fourier transforms of Ringel--Hall algebras and will be given in next section.

\begin{Lem} \label{value-xi(q,n)} For each prime power $q$ and positive integer $n$,
$$\sum_{\lz \vdash n}\frac{\prod_{s=1}^{\ell(\lz)-1}(1-q^s)}{a_\lz(q)}
=\frac{1}{q^n-1}.$$
\end{Lem}

As a consequence of the lemma and Theorem \ref{Main-thm-1}, we obtain the following
result which strengthens \cite[Cor.~6.10]{Hen} and covers Theorem \ref{Main-thm-2}
given in the introduction.

\begin{Thm} \label{prim-elt-space-basis}
Let $n\geq 1$ and fix $x\in \bbP_{\bbF_q,\,\sz}^1$ with $n=m\deg(x)$ for some $m\geq 1$.
Then the set
$$\bigg\{ p_m(x)-\frac{q^{te_y}-1}{q^{me_x}-1}p_t(y)\;\bigg|\; y\in
\bbP_{\bbF_q,\;\sz}^1,\; y\not=x,\;\text{ and }\;\,t\deg(y)=n\bigg\}$$
is a basis of $\cH_v(Q,\sz)^{\prim}_{n\dz}$. If, moreover,
$\sz=\id$, then the set
$$\bigg\{ p_m(x)-p_t(y)\;\bigg|\; y\in \bbP_{\bbF_q}^1,\; y\not=x,\;
\text{ and }\;\,t\deg(y)=n\bigg\}$$
is a basis of $\cH_v(Q)^{\prim}_{n\dz}$.
\end{Thm}

\medskip

\section{The proof of Lemma \ref{value-xi(q,n)}}

This section is devoted to proving Lemma \ref{value-xi(q,n)}. The strategy is to
apply Fourier transforms on Ringel--Hall algebras of the Kronecker quiver and
the cyclic quiver with two vertices.

%We first recall the notion of Fourier transforms for Ringel--Hall algebras
%of quivers given in \cite{L90,SVd2}.
%Analogously, we can define another Fourier transform $\ol{f}$ by
%\begin{equation*}
%  \ol{f}(u,w)=\wh{f}(u,-w)=q^{-{\rm dim} V/2}\sum_{v\in V} f(u,v)\overline{\psi}(\lr{v,w}),
%\end{equation*}
%where $\overline{\psi}(\lr{v,w})$ is the complex conjugate of ${\psi}(\lr{v,w})$.
%We have the following well-known properties about $\wh{f}$ and $\ol{f}$ which will be used later.
%
%\begin{Lem}[{\cite[Lemma 5.1]{SVd1}}]\label{Lem:Fourier_property}
%{\rm (1)} $\ol{\wh{f}}=\wh{\ol{f}}=f$, i.e., $\ol{f}$ is the Fourier inverse transform of $\wh{f}$.
%%    \item[(2)] If $U,V,W$ are, respectively, a $H$-set and $H$-representations for a finite group $H$, and if $\lr{-,-}$ is $H$-invariant, then $\widehat{hf}=h\hat{f}$ for $h\in H$,
%
%{\rm (2)} $(\wh{f},\,\ol{g})=(f,g)$.
%\end{Lem}
Let $Q=(Q_0=I,Q_1)$ be a finite quiver. In order to define Fourier transforms of Ringel--Hall
algebras, we give an alternative definition of the Ringel--Hall algebra $\cH_v(Q)$ in terms
of complex valued functions over representation spaces of $Q$ due to Lusztig \cite{L90}.
More precisely, for each $\bfd=(d_i)_i\in\bbN I$, let $V=\oplus_{i\in I}V_i$ be an $I$-graded
vector space of dimension vector $\bfd$ over $\bbF_q$ and define
$$E_V=E_{Q,V}:=\prod_{\rho\in Q_1}\Hom_k(V_{\tail\rho},V_{\head\rho}).$$
Then a point $x=(x_\rho)_\rho$ of $E_V$ gives rise to a
representation $V(x)=(V(x)_i,V(x)_\rho)$ of $Q$ over $\bbF_q$ with $V(x)_i=V_i$ for
$i\in I$ and $V(x)_\rho=x_\rho$ for $\rho\in Q_1$. The group
$G_V:=\prod_{i\in I}{\rm GL}(V_i)$ acts on $E_V$ by conjugation
$$(g_i)_i\cdot(x_\rho)_\rho=(g_{\head\rho}x_\rho g_{\tail\rho}^{-1})_\rho,$$
and the $G_V$-orbits ${\mathfrak O}_x=G_V\!\cdot\! x$ in $E_V$ correspond bijectively to the
isoclasses $[V(x)]$ of representations of $Q$ with dimension vector $\bfd$.

Furthermore, let $\cF(Q)_V$ denote the set of all $G_V$-invariant functions
$E_V\ra\bbC$. Then $\cF(Q)_V$ is a complex vector space with a basis
all the characteristic functions $\chi_{\fkO_x}$ of $G_V$-orbits in $E_V$.
Since, up to isomorphism, the vector space $\cF(Q)_V$ only depends on $\bfd$, we simply
set $\cF(Q)_\bfd=\cF(Q)_V$. Finally, set
$$\cF(Q)=\bigoplus_{\bfd\in\bbN I}\cF(Q)_\bfd.$$
 Take $\bfd,\bfd',\bfd''\in\bbN Q_0$ with $\bfd=\bfd'+\bfd''$ and let $V, V',V''$ be
$I$-graded $\bbF_q$-vector spaces of dimension vectors $\bfd,\bfd',\bfd''$, respectively, satisfying
$V_i=V_i'\oplus V_i''$ for $i\in I$. Then for $f\in\cF(Q)_{\bfd'}=\cF(Q)_{V'}$ and
$g\in \cF(Q)_{\bfd''}=\cF(Q)_{V''}$, define their convolution product $f*g\in\cF(Q)_{\bfd}=\cF(Q)_V$
by setting
$$f*g(x)=v^{\lr{\bfd',\bfd''}}\sum_{U\subset V(x)}f(V(x)/U)g(U),\;\,\forall\,x\in E_V,$$
 where the sum is over all subrepresentations $U$ of $V(x)$. By \cite{L90},
$\cF(Q)$ is an associative algebra, and there is an algebra isomorphism
$$\cF(Q)\lra \cH_v(Q),\;\chi_{\fkO_x}\lmto [V(x)].$$
In what follows, we will simply identify $\cF(Q)$
with $\cH_v(Q)$. Thus, each element in $\cH_v(Q)$ is also viewed as a function
on the associated representation space, that is, for each $M\in\Rep_{\bbF_q}Q$, $[M]$
is viewed as the characteristic function $\chi_{[M]}=\chi_{\fkO_M}$, and in general,
each $x=\sum_{[M]} \lz_M[M]\in \cH_v(Q)$ is identified with the function
$\sum_{[M]}\lz_M\chi_{[M]}$.

Now we are ready to define Fourier transforms. Let $E$ be a subset of
$Q_1$ and denote by $Q'$ the quiver obtained from $Q$ by reversing all
the arrows in $E$. Given $\bfd\in\bbN I$, let $V=\oplus_{i\in I}V_i$ be an $I$-graded
$\bbF_q$-vector space of dimension vector $\bfd$ and set
$$X_V=\prod_{\rho\in Q_1\backslash E}\Hom_k(V_{\tail\rho},V_{\head\rho}), \;
Y_V=\prod_{\rho\in E}\Hom_k(V_{\tail\rho},V_{\head\rho}),\;
Y'_V=\prod_{\rho\in E}\Hom_k(V_{\head\rho},V_{\tail\rho}).$$
Then
$$E_{Q,V}=X_V\times Y_V \;\text{ and }\; E_{Q',V}=X_V\times Y_V',$$
 and there is a non-degenerate bilinear form defined by
$$\lr{-,-}:Y_V\times Y_V'\lra \bbF_q,\;((C_\rho),(D_\rho))\longmapsto
\sum_{\rho\in E}\tr(C_\rho D_\rho).$$
 Let us fix a non-trivial additive character $\psi:\bbF_q\ra \bbC^\times$.
For each $f\in\cF(Q)_\bfd=\cF(Q)_V$, that is, a $G_V$-invariant function
$f:E_{Q,V}=X_V\times Y_V\ra\bbC$, define
$${\wh f}:E_{Q',V}=X_V\times Y'_V\lra\bbC$$
in $\cF(Q')_\bfd$ by setting
$${\wh f}(x,y')=q^{-{\rm dim} Y_V/2}\sum_{y\in Y_V}f(x,y)\psi(\lr{y,y'}),\;\forall\,
(x,y')\in X_V\times Y'_V.$$
 This gives a $\bbC$-linear map
$$\Phi_\bfd: \cF(Q)_V\lra \cF(Q')_V,\;f\longmapsto \wh f.$$
Finally, we put
$$\Phi_{Q',Q}=(\Phi_\bfd)_\bfd:\cF(Q)=\bigoplus_{\bfd\in\bbN I}\cF(Q)_\bfd\lra
\bigoplus_{\bfd\in\bbN I}\cF(Q')_\bfd=\cF(Q'),$$
 called the Fourier transform.

%Similarly, by defining ${\ol f}:E_{Q',V}=X_V\times Y'_V\ra\bbC$
%in $\cF(Q')_\bfd$ via
%$${\ol f}(x,y')=q^{-{\rm dim} Y_V/2}\sum_{y\in Y_V}f(x,y)\oline{\psi}(\lr{y,y'}),$$
%we obtain a $\bbC$-linear map $\ol\Phi_\bfd: \cF(Q)_V\ra \cF(Q')_V,\;f\mapsto \ol f$,
%and thus, a map
%$$\ol\Phi_{Q',Q}=(\ol\Phi_\bfd)_\bfd:\cF(Q)=\bigoplus_{\bfd\in\bbN I}\cF(Q)_\bfd\lra
%\bigoplus_{\bfd\in\bbN I}\cF(Q')_\bfd=\cF(Q').$$
%\begin{Rem} By the definition, both $\Phi_{Q',Q}$ and $\ol\Phi_{Q',Q}$ depend on the
%choice of $\psi$. Thus, we sometimes write $\Phi_{Q',Q}^\psi=\Phi_{Q',Q}$ and
%$\ol\Phi_{Q',Q}^\psi=\ol\Phi_{Q',Q}$ in order to emphasize the role of $\psi$.
%
%\end{Rem}
By identifying $\cF(Q)$ and $\cF(Q')$ with
$\cH_v(Q)$ and $\cH_v(Q')$, respectively, we view $\Phi_{Q',Q}$ as
a map $\cH_v(Q)\ra\cH_v(Q')$. The following lemma is stated in \cite[\S13]{L90}; see
\cite[Lem.~5.1\ \& Prop.~7.1]{SVd1}
for a proof.

\begin{Lem} \label{Lusztig-SV} Keep all the notation above. Then
$\Phi_{Q',Q}: \cH_v(Q)\ra\cH_v(Q')$ is an isomorphism of $\bbC$-algebras.
% with the inverse $\ol\Phi_{Q,Q'}: \cH_v(Q')\ra\cH_v(Q)$.
\end{Lem}

\begin{Rem} Fourier transforms can also be defined on Ringel--Hall
algebras of $\bbF_q$-algebras associated with quivers with automorphism;
see \cite{Ma-C}.
\end{Rem}

Let us look at a simple example of the Fourier transform
$$\Phi_{Q',Q}: \cH_v(Q)\lra \cH_v(Q'),$$
where the quivers $Q$ and $Q'$ are given by
\begin{center}
\setlength{\unitlength}{1mm}
\begin{picture}(100,10)
 \put(15,5){\circle*{1}} \put(30,5){\circle*{1}}
 \put(70,5){\circle*{1}} \put(85,5){\circle*{1}}
\put(16,5){\vector(1,0){13}} \put(84,5){\vector(-1,0){13}}

\put(0,5){$Q:$} \put(55,5){$Q':$}
\put(11,4){$1$} \put(32,4){$2$,}
\put(66,4){$1$} \put(87,4){$2$.}
\put(21,5.8){$\az$} \put(76,5.8){$\az'$}
\end{picture}
\end{center}
 Let $S_i$ and $P_i$ denote the simple and projective representation of
$Q$ corresponding to vertex $i$, respectively, while let $S'_i$ and
$P'_i$ denote the simple and projective representation of
$Q'$ corresponding to vertex $i$, respectively. Then
$\Phi_{Q',Q}([S_i])=[S_i']$ for $i=1,2$, and
$$\aligned
\Phi_{Q',Q}([P_1])&=\Phi_{Q',Q}(v[S_1][S_2]-[S_2][S_1])=v[S'_1][S'_2]-[S'_2][S'_1]\\
&=-v^{-1}[P'_2]+(v-v^{-1})[S_1'\oplus S_2']).
\endaligned$$
Hence, for each $n\geq 1$,
$$\Phi_{Q',Q}([P_1]^n)=\big( \Phi_{Q',Q}([P_1])\big)^n=(-1)^nv^{-n}[P_2']^n+\text{other terms}.$$
Since
$$[nP_1]=\frac{v^{-n(n-1)}}{[n]!}[P_1]^n\;\text{ and }\;\, [n P_2']=\frac{v^{-n(n-1)}}{[n]!}[P'_2]^n,$$
where $[n]!=[1][2]\cdots [n]$ with $[s]=(v^s-v^{-s})/(v-v^{-1})$ for $1\leq s\leq n$,
 we have
$$\big(\Phi_{Q',Q}([nP_1])\big)([nP_2'])=(-1)^nv^{-n}.$$

On the other hand, for each matrix $X\in \bbF_q^{n\times n}$, let $V(X)$
be the representation of $Q$ given by
$$V(X)_1=V(X)_2=\bbF_q^n\;\text{ and }\;\, V(X)_\az=X.$$
Then $V(X)\cong nP_1$ if and only if $X$ is invertible. Since
$$n P_2'\cong\; \bbF_q^n \stackrel{I_n}{\longleftarrow} \bbF_q^n,$$
where $I_n$ is the identity matrix of size $n$, it follows from the definition
of $\Phi_{Q',Q}$ that
$$\aligned
\big(\Phi_{Q',Q}([n P_1])\big)([n P_2'])&=q^{-n^2/2}\sum_{X \in {\bbF_q}^{n \times n}}
\chi_{[nP_1]}([V(X)])\psi(\tr (XI_n))\\
  &=q^{-n^2/2}\sum_{X \in \GL_n(\bbF_q)}\psi(\tr\,X),
\endaligned$$
where $\psi$ is the non-trivial additive character of $\bbF_q$ arising in
the definition of $\Phi_{Q',Q}$.

Consequently, we obtain the equality
\begin{equation} \label{Fourier-transf-A_2}
\sum_{X \in \GL_n(\bbF_q)}\psi(\tr\, X)=(-1)^nq^{n(n-1)/2},
\end{equation}
which is useful later on.

\bigskip

\noindent{\bf Proof of Lemma \ref{value-xi(q,n)}.} Consider the Fourier transform
$$\Phi:=\Phi_{C_2,K_2}: \cH_v(K_2)\lra \cH_v(C_2)$$
associated with the Kronecker quiver $K_2$ and the cyclic quiver $C_2$ of two vertices:
\begin{center}
\setlength{\unitlength}{1mm}
\begin{picture}(100,15)
 \put(15,6){\circle*{1}} \put(35,6){\circle*{1}}
 \put(70,6){\circle*{1}} \put(90,6){\circle*{1}}
\put(16,7){\vector(1,0){18}} \put(16,5){\vector(1,0){18}}
\put(71,7){\vector(1,0){18}} \put(89,5){\vector(-1,0){18}}
\put(0,5){$K_2:$} \put(55,5){$C_2:$}
\put(11,5){$1$} \put(37,5){$2$} \put(66,5){$1$} \put(92,5){$2$}
\put(24,8){$\az$} \put(24,1){$\bz$} \put(79,8){$\az$} \put(79,1){$\bz'$}
\end{picture}
\end{center}
which is defined by inverting the arrow $\bz$ in $K_2$. Note that
$$\cF_{K_2}=\cF_{C_2}=\{m\dz\mid m\geq 1\}\;\text{ with $\dz=(1,1)$.}$$

By the definition, for given representations $M$ and $N$ of $K_2$
and $C_2$, respectively, which are depicted as follows:
\begin{center}
\setlength{\unitlength}{1mm}
\begin{picture}(100,14)
  \put(-1,5){$M:$} \put(53,5){$N:$}
  \put(16,7){\vector(1,0){18}} \put(16,5){\vector(1,0){18}}
  \put(10,5){$\bbF_q^m$} \put(35,5){$\bbF_q^n$,}
  \put(23,9){$X_1$} \put(23,0){$X_2$}
  \put(71,7){\vector(1,0){18}} \put(89,5){\vector(-1,0){18}}
  \put(65,5){$\bbF_q^m$} \put(90,5){$\bbF_q^n$,}
  \put(78,9){$Y_1$} \put(78,0){$Y_2$}
\end{picture}
\end{center}
we have
\begin{equation} \label{formula-Fourier}
\big(\Phi([M])\big)([N])=q^{-mn/2}\sum_{Z\in \bbF_q^{n\times m}}\chi_{[M]}
\big([L(Z)]\big)\psi(\tr(ZY_2)),
\end{equation}
 where $L(Z)$ is given by
\begin{center}
\setlength{\unitlength}{1mm}
\begin{picture}(60,14)
  \put(-5,5){$L(Z):$}
  \put(16,7){\vector(1,0){18}} \put(16,5){\vector(1,0){18}}
  \put(10,5){$\bbF_q^m$} \put(35,5){$\bbF_q^n$.}
  \put(23,9){$Y_1$} \put(23,0){$Z$}
\end{picture}
\end{center}

For each $m\geq 1$, consider the following indecomposable representations of $K_2$:
\begin{center}
\setlength{\unitlength}{1mm}
\begin{picture}(100,12)
  \put(-2.8,5){$I_m(0)\!:$} \put(50,5){$I_m(\infty)\!:$}
  \put(16,7){\vector(1,0){18}} \put(16,5){\vector(1,0){18}}
  \put(10,5){$\bbF_q^m$} \put(35,5){$\bbF_q^m$,}
  \put(23,9){$I_m$} \put(23,0){$J_m$}
  \put(71,7){\vector(1,0){18}} \put(71,5){\vector(1,0){18}}
  \put(65,5){$\bbF_q^m$} \put(90,5){$\bbF_q^m$,}
  \put(78,9){$J_m$} \put(78,0){$I_m$}
\end{picture}
\end{center}
 where $J_m$ is the nilpotent Jordan block of size $m$. Indeed, all the
indecomposables $I_m(0)$ (resp., $I_m(\infty)$), $m\geq 1$,
furnish a complete set of indecomposable representations in a homogeneous
tube $\cT_0$ (resp., $\cT_\infty$) in the Auslander--Reiten quiver of $\Rep_{\bbF_q}K_2$.
For each partition $\lz=(\lz_1,\ldots,\lz_t)\vdash n$, set
$$I_\lz(0)=\bigoplus_{i=1}^t I_{\lz_i}(0)\;\text{ and }\;\,
I_\lz(\infty)=\bigoplus_{i=1}^t I_{\lz_i}(\infty).$$
In other words, $I_\lz(0)$ and $I_\lz(\infty)$ have the form
\begin{center}
\setlength{\unitlength}{1mm}
\begin{picture}(100,14)
  \put(-2.8,5){$I_\lz(0)\!:$} \put(50,5){$I_\lz(\infty)\!:$}
  \put(16,7){\vector(1,0){18}} \put(16,5){\vector(1,0){18}}
  \put(10,5){$\bbF_q^n$} \put(35,5){$\bbF_q^n$,}
  \put(23,9){$I_n$} \put(23,1){$J_\lz$}
  \put(71,7){\vector(1,0){18}} \put(71,5){\vector(1,0){18}}
  \put(65,5){$\bbF_q^n$} \put(90,5){$\bbF_q^n$,}
  \put(78,9){$J_\lz$} \put(78,1){$I_n$}
\end{picture}
\end{center}
 where $J_\lz$ denotes the block diagonal matrix ${\rm diag}(J_{\lz_1},\ldots,J_{\lz_t})$.
Then attached to the tubes $\cT_0$ and $\cT_\infty$, we have two regular primitive
elements in $\cH_v(K_2)$
$$p_n(0)=\sum_{\lz \vdash n}\bigg(\prod_{s=1}^{\ell(\lz)-1}(1-q^s)\bigg)[I_\lz(0)]\;
\text{ and }\;\, p_n(\infty)=\sum_{\lz \vdash n}
\bigg(\prod_{s=1}^{\ell(\lz)-1}(1-q^s)\bigg)[I_\lz(\infty)]$$
whose difference
$$p_n(0)-p_n(\infty)=:p_n^{K_2}$$
 is a primitive element in $\cH_v(K_2)^\prim_{n\dz}$ by Theorem \ref{prim-elt-space-basis};
 see also \cite[Cor.~6.10]{Hen}.

According to \cite[Thm~3.5]{DM}, $\Phi(p_n^{K_2})$ is a primitive element in
$\cH_v(C_2)^\prim_{n\dz}$ which, by Proposition \ref{prim-elt-cyclic}, is a
$\bbC$-linear combination of the elements in the basis set
$$\{p_n^{(2)}\}\cup \{p_m(x)\mid 0\not=x\in \bbA_q^1\;\text{ and }\; m\deg(x)=n\}$$
of $\cH_v(C_2)^\prim_{n\dz}$.

As in Section 4, for each $i\in\{1,2\}$ and $m\geq 1$,
let $S_i[m]$ be the indecomposable nilpotent representation of $C_2$ of dimension $m$
with top $S_i$. We are going to compute
$$\Phi(p_n^{K_2})([nS_1[2]])\;\text{ and }\;\,\Phi(p_n^{K_2})([nS_2[2]]).$$

It is clear that $C_2$ admits an automorphism $\tau$ which swaps vertices $1$
and $2$, as well as arrows $\az$ and $\bz'$. Then $\tau$ induces a bialgebra
isomorphism
$$\tau: \cH_v^0(C_2)\lra \cH_v^0(C_2)$$
 taking particularly $[S_1[m]]\mapsto [S_2[m]]$ and $[S_2[m]]\mapsto [S_1[m]]$.
Then $\tau(p_n^{(2)})$ is again a primitive element in $\cH_v^0(C_2)$ of dimension
vector $n\dz$, which is thus a multiple of $p_n^{(2)}$. By \eqref{prim-elt-cyc-r},
$p_n^{(2)}$ has the form
$$p_n^{(2)}=[S_1[2n]]+[S_2[2n]]+\text{other terms}.$$
 This forces that $\tau(p_n^{(2)})=p_n^{(2)}$. Because $\tau([nS_1[2]])=[nS_2[2]]$, we also have
$$p_n^{(2)}([nS_1[2]])=p_n^{(2)}([nS_2[2]]).$$
 Since for each $0\not=x\in \bbA_q^1$ with $m\deg(x)=n$, $p_m(x)$ does not
support on $[nS_1[2]]$ and $[nS_2[2]]$, it follows that
$$\Phi(p_n^{K_2})([nS_1[2]])=\Phi(p_n^{K_2})([nS_2[2]]).$$
For notational simplicity, we write
$$M_1:=nS_1[2]\;\text{ and }\;\, M_2:=nS_2[2],$$
which clearly have the form
\begin{center}
\setlength{\unitlength}{1mm}
\vspace{-4mm}
\begin{picture}(100,22)
  \put(-5,9){$M_1$:} \put(50,9){$M_2$:}
  \put(16,11){\vector(1,0){18}} \put(34,9){\vector(-1,0){18}}
  \put(9,9){$\bbF_q^n$} \put(35,9){$\bbF_q^n$,}
  \put(23,13){$I_n$} \put(23,4){$0$}

  \put(71,11){\vector(1,0){18}} \put(89,9){\vector(-1,0){18}}
  \put(64,9){$\bbF_q^n$} \put(90,9){$\bbF_q^n$.}
  \put(78,13){$0$} \put(78,4){$I_n$}
\end{picture}
\end{center}
On the one hand, we have
$$\Phi(p_n^{K_2})([M_1])=\sum_{\lz \vdash n}\bigg(\prod_{s=1}^{\ell(\lz)-1}(1-q^s)\bigg)
\bigg(\Phi([I_\lz(0)])([M_1])-\Phi([I_\lz(\infty)])([M_1])\bigg).$$
Applying the formula in \eqref{formula-Fourier} shows that for
a partition $\lz \vdash n$, $\Phi([I_\lz(\infty)])([M_1])=0$ and that
$$\aligned
\Phi([I_\lz(0)])([M_1])&=q^{-n^2/2}\sum_{X\in\bbF_q^{n\times n}:X\sim J_\lz}\psi(0)
=q^{-n^2/2}|\{X \in \bbF_q^{n \times n}\mid X\sim J_\lz\}|\\
&=q^{-n^2/2}\frac{\lvert \GL_n(\bbF_q) \rvert}{a_\lz(q)}=q^{-n^2/2}
\prod_{i=0}^{n-1}(q^n-q^i)\frac{1}{a_\lz(q)},
\endaligned$$
 where $X\sim J_\lz$ means that $X$ and $J_\lz$ are conjugate in
 $\bbF_q^{n\times n}$. Therefore,
$$\aligned
\Phi(p_n^{K_2})([M_1])&=\sum_{\lz \vdash n}
\bigg(\prod_{s=1}^{\ell(\lz)-1}(1-q^s)\bigg) \Phi([I_\lz(0)])([M_1])\\
&=q^{-n^2/2}\sum_{\lz \vdash n}\bigg(\prod_{s=1}^{\ell(\lz)-1}(1-q^s)\bigg)
\bigg(\prod_{i=0}^{n-1}(q^n-q^i)\frac{1}{a_\lz(q)}\bigg)\\
  &=q^{-n^2/2}\prod_{i=0}^{n-1}(q^n-q^i)
\bigg(\sum_{\lz \vdash n}\frac{\prod_{s=1}^{\ell(\lz)-1}(1-q^s)}{a_\lz(q)}\bigg).
\endaligned$$

On the other hand, we have
$$\Phi(p_n^{K_2})([M_2])=\sum_{\lz \vdash n}\bigg(\prod_{s=1}^{\ell(\lz)-1}(1-q^s)\bigg)
\bigg(\Phi([I_\lz(0)])([X_2])-\Phi([I_\lz(\infty)])([M_2])\bigg).$$
Again, by applying the formula in \eqref{formula-Fourier}, we have
$$\aligned
& \Phi([I_\lz(0)])([M_2])=0\;\;\text{ for all $\lz\vdash n$, and}\\
& \Phi([I_\lz(\infty)])([M_2])\not=0 \;\;\text{ only if $\lz=(1^n)$.}
\endaligned$$
Hence,
$$\aligned
\Phi(p_n^{K_2})([M_2])&=-\prod_{i=1}^{n-1}(1-q^i)\Phi([I_{(1^n)}(\infty)])([M_2])\\
&=-q^{-n^2/2}\prod_{i=1}^{n-1}(1-q^i)\sum_{X \in \GL_n(k)}\psi(\tr X)\\
&=-q^{-n^2/2}\prod_{i=1}^{n-1}(1-q^i)(-1)^nq^{n(n-1)/2}   \qquad \text{by (\ref{Fourier-transf-A_2})}\\
&=q^{-n^2/2}\prod_{i=1}^{n-1}(q^n-q^i).
\endaligned$$
 We conclude that
$$q^{-n^2/2}\prod_{i=0}^{n-1}(q^n-q^i)\bigg(\sum_{\lz \vdash n}
\frac{\prod_{s=1}^{\ell(\lz)-1}(1-q^s)}{a_\lz(q)}\bigg)
=q^{-n^2/2}\prod_{i=1}^{n-1}(q^n-q^i),$$
that is,
$$\sum_{\lz \vdash n}\frac{\prod_{s=1}^{\ell(\lz)-1}(1-q^s)}{a_\lz(q)}
=\frac{1}{q^n-1}.$$
This finishes the proof.

\begin{Rem} It might be interesting to compare the identity above with the following two identities
$$\sum_{\lz \vdash n}\frac{\big(\prod_{s=1}^{\ell(\lz)-1}(1-q^s)\big)^2}
{a_\lz(q)}=\frac{n}{q^n-1}$$
and
$$\sum_{\lz \vdash n}\frac{1}{a_{\lz}(q)}=\frac{q^{n(n-1)/2}}{(q-1)(q^2-1)\cdots (q^n-1)},$$
which are given in \cite[Ch.~III]{Mac} and \cite[Ex.~I]{Hua}, respectively.
\end{Rem}

\bigskip

{\bf{Acknowledgements:}} The authors are grateful to Jiuzhao Hua for helpful discussions.
This work was supported by the National Natural Science Foundation of China
(Grant No. 12031007).


\bigskip

\bigskip

\begin{thebibliography}{99}

\bibitem{BG} A. Berenstein and J. Greenstein, {\em Primitively generated Hall algebras},
Pacific J. Math. {\bf 281} (2016), 287--331.

%\bibitem{BGP} I.~N. Bernstein, I.~M. Gelfand and V.~A. Ponomarev,
%{\em Coxeter functors and Gabriel's theorem}, Uspehi Math. Nauk
%{\bf 28} (1973), 19--33.

\bibitem{Bor} R. Borcherds, {\em Generalized Kac--Moody algebras}, J. Algebra
{\bf 115} (1988), 501--512.

\bibitem{BS} T. Bozec and O. Schiffmann, {\em Counting absolutely cuspidals for quivers},
Math. Z. {\bf 292} (2019), 133--149.

\bibitem{CB} W. Crawley-Boevey, {\em Exceptional sequences of representations of quivers},
Canad. Math. Soc. Conference Proc. {\bf 14} (1993), 117--124.

\bibitem{DD05} B. Deng and J. Du, {\em Monomial bases for quantum affine
${\frak sl}_n$}, Adv. Math. {\bf 191} (2005), 276--304.

\bibitem{DD06}  B. Deng and J. Du, {\em Frobenius morphisms and representations
of algebras}, Trans. Amer. Math. Soc. {\bf 358} (2006), 3591--3622.

\bibitem{DDF}
B. Deng, J. Du and Q. Fu, {\em A Double Hall Algebra Approach to Affine Quantum
Schur--Weyl Theory}, London Math. Soc. Lecture Note Series 401, Cambridge, 2012.

\bibitem{DDPW} B. Deng, J. Du, B. Parshall and J. Wang, {\em Finite Dimensional
Algebras and Quantum Groups}, Mathematical Surveys and Monographs
Volume 150, Amer. Math. Soc., Providence 2008.

\bibitem{DM} B. Deng and C. Ma, {\em Fourier transforms on Ringel--Hall algebras},
Algebr. Represent. Theory {\bf 26} (2023), 1913--1930.

\bibitem{DOP} B. Deng and K. Obul and Y. Pang, {\em Representations of quivers
with automorphisms over finite fields} (in Chinese), Sci. Sin. Math. {\bf 47}, (2017), 1--20.

%\bibitem{DX1} B. Deng and J. Xiao, {\em On double Ringel--Hall
%algebras}, J. Algebra {\bf 251} (2002), 110--149.

%\bibitem{DX2} B. Deng and J. Xiao, {\em Ringel--Hall algebras and Lusztig's
%symmetries}, J. Algebra {\bf 255} (2002), 357--372.

\bibitem{DX3} B. Deng and J. Xiao, {\em A new appraoch to Kac's
theorem on representations of valued quivers},  Math. Z. {\bf 245}
(2003), 183--199.

%\bibitem{DXZ} B. Deng, J. Xiao and M. Zhao, {\em Canonical bases and $q$-deformed Fock spaces},
%in preparation.

\bibitem{DR} V. Dlab and C.~M. Ringel, {\em Indecomposable
Representations of Graphs and Algebras}, Memoirs Amer. Math. Soc.,
no.~173, Amer. Math. Soc., Providence, 1976.

\bibitem{DF} P.~W. Donovan and M.~R. Freislich,
{\em The Representation Theory of Finite Graphs and Associated
Algebras}, Carleton Math. Lecture Notes, no.~5, 1973.


\bibitem{GR92} P. Gabriel and A.~V. Roiter, {\em
Representations of Finite-Dimensional Algebras}, with a chapter by
B. Keller, Encyclopaedia Math. Sci., no.~73, Algebra, VIII,
Springer-Verlag, Berlin, 1992.

\bibitem{GL} W. Geigle and H. Lenzing, Perpendicular categories with applications
to representations and sheaves, J. Algebra {\bf 144} (1991), 273--343.

\bibitem{Gr95} J.~A. Green, {\em Hall algebras, hereditary algebras and quantum groups},
Invent. Math. {\bf 120} (1995), 361--377.

%\bibitem{Guo} J.~Y. Guo, {\em The Hall polynomials of a cyclic serial algebra},
%Comm. Algebra {\bf 23} (1995), 743--751.

\bibitem{Ha} P. Hall, {\em The algebra of partitions}, in: {\em Proceedings of the 4th
Canadian Mathematical Congress, Banff 1957}, University of Toronto
Press, Toronto, 1959, pp.~147--159.

\bibitem{Hen} L. Hennecart, {\em Isotropic cuspidal functions in the Hall algebra of a quiver},
Int. Math. Res. Not. IMRN {\bf 2021}, no. 15, 11514--11564.

\bibitem{Hua} J. Hua, {\em Counting representations of quivers over finite fields},
J. Algebra {\bf 226} (2000), 1011--1033.

\bibitem{HX} J. Hua and J. Xiao, {\em On Ringel--Hall algebras of tame
hereditary algebras}, Algebr. Represent. Theory {\bf 5} (2002), 527--550.

\bibitem{Hub1} A. Hubery, {\em Symmetric functions and the center of
the Ringel--Hall algebra of a cyclic quiver}, Math. Z. {\bf 251} (2005), 705--719.

\bibitem{Hub2} A. Hubery, {\em Three presentations of the Hopf algebra
${\mathcal U}_v({\widehat{\frak{gl}}_n})$}, preprint, 2009.

\bibitem{Hub3} A. Hubery, {\em Ringel--Hall algebras of cyclic quivers}, S$\tilde{\rm a}$o
Paulo J. Math. Sci. 4 (2010), no. 3, 351--398.

\bibitem{Kac90} V. Kac, {\em Infinite Dimensional Lie Algebras}, 3rd ed.,
Cambridge University Press, 1990.

%\bibitem{Kan} S. J. Kang, {\em Quantum deformations of generalized
%Kac-Moody algebras and their modules}, J. Algebra {\bf 175} (1995), 1041--1066.

\bibitem{L90} G. Lusztig, {\em Canonical bases arising from quantized enveloping algebras},
J. Amer. Math. Soc. {\bf 3} (1990), 447--498.

\bibitem{L99} G. Lusztig, {\em Aperiodicity in quantum affine $\frak{gl}_n$},
Asian J. Math. {\bf 3}(1999), 147--177.

%\bibitem{L93} G. Lusztig, {\em Introduction to Quantum
%Qroups}, Progress in Mathematics 110, Birkh\"{a}user Boston, Inc., Boston, MA, 1993.

\bibitem{Ma-C} C. Ma, {\em Fourier transforms and Ringel--Hall algebras
of valued quivers}, Proc. Amer. Math. Soc. {\bf 149} (2021), no. 5, 1875--1887.

\bibitem{Ma} R. Ma, {\em Primitive elements in the Hall algebra of a cyclic
quiver}, arXiv:2403.18065.

\bibitem{Mac} I.~G. Macdonald,
{\em Symmetric Functions and Hall Polynomials}, 2nd ed., Clarendon
Press, Oxford, 1995.

\bibitem{Obul} A. Obul, {\em Minimal generating system of Ringel-Hall algebras of
affine valued quivers}, J. Algebra {\bf 297} (2006), 311--332.

\bibitem{Na} L.~A. Nazarova, {\em Representations of quivers of infinite type},
Math. USSR Izvestija Ser. Mat. {\bf 7}(1973), 752--791.

%\bibitem{Rein} M. Reineke, {\em Counting rational points of quiver moduli},
%Int. Math. Res. Not. 2006, Art. ID 70456, 19 pp.

\bibitem{R76} C.~M. Ringel, {\em Representations of K-species and bimodules},
J. Algebra {\bf 41} (1976), 269--302.

\bibitem{R84} C.~M. Ringel, {\em Tame Algebras and Integral Quadratic
Forms}, Lecture Notes in Mathematics, no.~1099, Springer-Verlag,
Berlin, 1984.

\bibitem{R90a} C.~M. Ringel, {\em Hall algebras},
in: {\em Topics in Algebra, Part 1}, S. Balcerzyk {\it et al.}
(eds.), Banach Center Publications, no.~26, 1988, pp.~433--447.

%\bibitem{R90b} C.~M. Ringel, {\em Hall polynomials for the
%representation-finite hereditary algebras}, Adv. Math. {\bf
%84}(1990), 137--178.

\bibitem{R90c} C.~M. Ringel, {\em Hall algebras and quantum groups},
Invent. Math. {\bf 101} (1990), 583--592.

\bibitem{R93a} C.~M. Ringel, {\em The composition algebra of a cyclic
quiver, Towards an explicit description of the quantum groups of
type $\tilde A_n$}, Proc. London Math. Soc. {\bf 66}(1993),
507--537.

\bibitem{R93} C.~M. Ringel,
{\em Hall algebras revisited}, in {\em Quantum Deformations of
Algebras and Their Representations}, A. Joseph \&\ S. Shnider
(eds.), Israel Mathematical Conference Proceedings, no.~7, Bar-Ilan
University, Bar-Ilan, 1993, pp.~171--176.

\bibitem{R94} C.~M. Ringel, {\em The braid group action on the set of exceptional
sequences of a hereditary artin algebra}, In: Contemp. Math. {\bf 171}, Amer. Math. Soc.,
Providence, 1994, pp. 339--352.

%\bibitem{R96} C.~M. Ringel, {\em Green's theorem on Hall algebras}, Representation theory
%of algebras and related topics (Mexico City, 1994), 185--245, CMS Conf. Proc., 19, Amer. Math. Soc.,
%Providence, RI, 1996.


\bibitem{SVd1} B. Sevenhant and M. Van den Bergh, {\em On the double of the
Hall algebra of a quiver}, J. Algebra {\bf 221} (1999), 135--160.

\bibitem{SVd2} B. Sevenhant and M. Van den Bergh, {\em A relation between
a conjecture of Kac and the structure of the Hall algebra},  J. Pure
Appl. Algebra {\bf 160} (2001), 319--332.

\bibitem{Sch} O. Schiffmann,
{\em The Hall algebra of a cyclic quiver and canonical bases of
Fock spaces}, Internat. Math. Res. Notices {\bf 8} (2000), 413--440.

\bibitem{Scho} A. Schofield, Semi-invariants of quivers, J. London Math. Soc. {\bf 43} (1991),
383--395.

\bibitem{St} E. Steinitz, {\em Zur Theorie der Abel'schen Gruppen},
Jahrsber. Deutsch. Math-Verein. {\bf 9}(1901), 80--85.

%\bibitem{X97} J. Xiao, {\em Drinfeld double and Ringel--Green theory of Hall algebras},
%J. Algebra {\bf 190} (1997), 100--144.

\bibitem{VV} M. Varagnolo and E. Vasserot, {\em On the
decomposition matrices of the quantized Schur algebra}, Duke Math.
J. {\bf 100}(1999), 267--297.

\bibitem{Zel} A. V. Zelevinsky, {\em Representations of Finite
Classical Groups}, Lecture Notes in Mathematics, no.~869,
Springer-Verlag, Berlin-New York, 1981.

\bibitem{Zh-p} P. Zhang, {\em Triangular decomposition of the composition algebra of
the Kronecker algebra}, J. Algebra {\bf 184} (1996), 159--174.

\bibitem{ZZG} P. Zhang, Y. Zhang and J. Guo {\em Minimal generators of Ringel--Hall
algebras of affine quivers}, J. Algebra {\bf 239} (2001), 675--204.

\end{thebibliography}
\end{document}